\documentclass[a4paper]{amsart}
\usepackage{microtype}
\usepackage[utf8]{inputenc}

\usepackage{amsmath,amstext,amssymb,mathrsfs,amscd,amsthm,indentfirst, bbm}
\usepackage{amsfonts}

\usepackage{caption}
\usepackage{subcaption}

\usepackage{float}
\usepackage{adjustbox}
\usepackage{booktabs}
\usepackage{verbatim}
\usepackage{enumerate}
\usepackage{graphicx}
\usepackage{textcomp}
\usepackage{tikz,tikz-cd}
\usetikzlibrary{patterns}
\usepackage{ifthen}
\usepackage{cite}
\usepackage[pagebackref]{hyperref}
\PassOptionsToPackage{hyphens}{url}\usepackage{hyperref}
\hypersetup{
    colorlinks  = true,
    linkcolor   = [rgb]{.20,.60,.22},
    citecolor   = [rgb]{0,.40,.80},
    urlcolor    = [rgb]{0,.40,.80}
}

\newcommand{\bQ}{\mathbb Q}
\newcommand{\bN}{\mathbb N}
\newcommand{\bZ}{\mathbb Z}
\newcommand{\bR}{\mathbb R}
\newcommand{\bC}{\mathbb C}

\newcommand{\bP}{\mathbb P}
\newcommand{\bS}{\mathbb S}
\newcommand{\cA}{\mathcal A}
\newcommand{\cB}{\mathcal B}

\newcommand{\cE}{\mathcal E}
\newcommand{\cS}{\mathcal S}
\newcommand{\cC}{\mathcal C}

\newcommand{\cN}{\mathcal N}
\newcommand{\cL}{\mathcal L}

\newcommand{\cO}{\mathcal O}

\newcommand{\cI}{\mathcal I}

\newcommand{\fm}{\mathfrak m}
\newcommand{\fX}{\mathfrak X}
\newcommand{\fY}{\mathfrak Y}
\newcommand{\fT}{\mathfrak T}

\newcommand{\lra}{\longrightarrow}

\newcommand{\fS}{\mathfrak S}

\DeclareMathOperator*{\colim}{colim}

\newcommand\map{\mathrm{map}}
\newcommand{\Sq}{\mathrm{Sq}}
\newcommand{\Sph}{\mathrm{Sph}}

\newcommand{\op}{\text{op}}

\usepackage{amsthm}
\theoremstyle{plain}

\newtheorem{theorem}{Theorem}[section]

\newtheorem{proposition}[theorem]{Proposition}
\newtheorem{lemma}[theorem]{Lemma}

\newtheorem{corollary}[theorem]{Corollary}

\theoremstyle{definition}
\newtheorem{definition}[theorem]{Definition}

\newtheorem{example}[theorem]{Example}
\newtheorem*{example*}{Example}
\newtheorem*{examples*}{Examples}

\newtheorem*{claim*}{Claim}
\newtheorem{remark}[theorem]{Remark}
\newtheorem*{remark*}{Remark}

\newtheorem*{question*}{Question}

\newtheorem*{theorem*}{Theorem}

\newcommand{\catS}{\mathsf{S}}
\newcommand{\catSp}{\mathsf{Sp}}
\newcommand{\catTop}{\mathsf{Top}}
\newcommand{\catF}{\mathsf{F}}
\newcommand{\catFun}{\mathsf{Fun}}
\newcommand{\BMH}{H^{BM}}
\newcommand{\CCH}{\check{H}_{c}}

\newcommand{\vertiii}[1]{{\left\vert\kern-0.25ex\left\vert\kern-0.25ex\left\vert #1 
    \right\vert\kern-0.25ex\right\vert\kern-0.25ex\right\vert}}

\let\emptyset\varnothing

\newcommand{\Gammahol}{\Gamma_{\mathrm{hol}}}

\newcommand{\Conf}{\mathrm{Conf}}
\newcommand{\rank}{\mathrm{rk}}

\newcommand{\Str}{\mathrm{Str}}
\newcommand{\cohodim}{\mathrm{cohodim}}

\renewcommand{\epsilon}{\varepsilon}

\newenvironment{forcedcentertikzcd}
 {\begin{lrbox}{\forcedcentertikzcdbox}\begin{tikzcd}}
 {\end{tikzcd}\end{lrbox}\makebox[0pt]{\usebox{\forcedcentertikzcdbox}}}
\newsavebox{\forcedcentertikzcdbox}

\newcommand\ColouredRectangle[2]{
    \fill[rectangle, color=lime, opacity=0.3] ({#1},{#2}) -- +(0, 1) -- +(1, 1) -- +(1,0) -- cycle;
}

\title[An h-principle for complements of discriminants]{An h-principle for complements of discriminants}

\author[Alexis Aumonier]{Alexis Aumonier}
\address{\parbox{\linewidth}{Department of Mathematical Sciences, University of Copenhagen,\\
Universitetsparken 5, 2100 København Ø, Denmark\\}}
\email{alexis-aumonier@math.ku.dk}

\begin{document}
\maketitle

\begin{abstract}
We compare spaces of non-singular algebraic sections of ample vector bundles to spaces of continuous sections of jet bundles. Under some conditions, we provide an isomorphism in homology in a range of degrees growing with the jet ampleness. As an application, when $\cL$ is a very ample line bundle on a smooth projective complex variety, we prove that the rational cohomology of the space of non-singular algebraic sections of $\cL^{\otimes d}$ stabilises as $d \to \infty$ and compute the stable cohomology. We also prove that the integral homology does not stabilise using tools from stable homotopy theory.
\end{abstract}


\section{Introduction}

The purpose of this paper is to study spaces of non-singular holomorphic sections of vector bundles by comparing them to spaces of continuous sections of appropriate jet bundles. The latter are particularly amenable to computations using tools from homotopy theory.

Given a holomorphic line bundle $\cL$ on a smooth projective complex variety $X$, one may consider the vector space of all holomorphic global sections $\Gammahol(X; \cL)$. To each section $s \in \Gammahol(X;\cL)$ is associated a geometric object: its vanishing set
\[
    V(s) := \left\{ x \in X \mid s(x) = 0 \right\} \subset X,
\]
and $s$ is called non-singular whenever its derivative $ds \in \Gammahol(\Omega^1_X \otimes \cL)$ does not vanish on $V(s)$. This implies in particular that $V(s)$ is a smooth subvariety of $X$. It has been known for a century now that when $\cL$ is a very ample line bundle, Bertini theorem implies that a generic section is non-singular. There is thus a Zariski open subset
\[
    \Gamma_{\mathrm{hol,ns}}(X;\cL) \subset \Gammahol(X;\cL)
\]
consisting of those non-singular sections, which geometrically can be interpreted as a moduli space of equations of certain smooth hypersurfaces in $X$. A prime example being the space $\Gamma_{\mathrm{hol,ns}}(\bC\bP^n;\cO(d))$ (sometimes modded out by $\bC^*$ or $GL_{n+1}(\bC)$) of smooth hypersurfaces of degree $d$ in the complex projective space $\bC\bP^n$.

The cohomology ring of $\Gamma_{\mathrm{hol,ns}}(X;\cL)$, sometimes known as the ring of characteristic classes, is therefore an important object in the study of hypersurface bundles. In this article, we give a way of computing it in a range.

\bigskip

Before revealing our main theorem, we will extend the classical situation above in two directions. To begin, instead of limiting ourselves to line bundles, we will look at sections of bundles of possibly higher rank. Furthermore, we observe that being non-singular imposes conditions on the value and derivative of a global section. We will generalise this situation by looking at a broader class of conditions on higher order derivatives, thus encompassing various other flavours of moduli spaces: hypersurfaces with simple nodes, smooth complete intersections, etc. (Although explicit computations of cohomology rings will only appear in forthcoming work.)

Having said this, let $X$ be a smooth projective complex variety and $\cE$ be a holomorphic vector bundle on $X$. One can construct a new holomorphic vector bundle $J^r\cE$ called the $r$-th jet bundle of $\cE$ together with a map on global sections $j^r \colon \Gammahol(\cE) \to \Gammahol(J^r\cE)$. Intuitively, for a section $s$ of $\cE$, the associated section $j^r(s)$ of the jet bundle records all derivatives of $s$ up to order $r$. For $\fT \subset J^r\cE$ a subset which we think of as ``forbidden derivatives'', we say that a section $s$ of $\cE$ is non-singular if $j^r(s)(x) \not\in \fT$ for all $x \in X$. For instance, when $\cE$ is a line bundle and $\fT \subset J^1\cE$ is the zero section, we recover the classical notion of non-singular sections discussed at the beginning of this article.
\begin{theorem}[see Theorem~\ref{thm:mainthm} for full generality]
Let $X$ be a smooth complex projective variety and $\cE$ be a holomorphic vector bundle on it. Let $r \geq 0$ be an integer and $\fT \subset J^r\cE$ be a closed subvariety of the $r$-th jet bundle of $\cE$ of codimension at least $\dim_\bC X + 1$. We write
\[
    \Gamma_{\mathrm{hol,ns}}\left( \cE \right) := \left\{ s \in \Gammahol(\cE) \mid \forall x \in X \ j^r(s)(x) \not\in \fT \right\}
\]
for the space of non-singular holomorphic sections of $\cE$. If $\cE$ is $d$-jet ample, the composition
\[
    \Gamma_{\mathrm{hol,ns}}\left( \cE \right) \overset{j^r}{\lra} \Gamma_{\mathrm{hol}}\left( J^r\cE - \fT \right) \hookrightarrow \Gamma_{\cC^0}\left( J^r\cE - \fT \right)
\]
induces an isomorphism in integral homology in the range of degrees $* < \frac{d-r}{r+1}$.
\end{theorem}

The theorem above can be strengthened, and in Section~\ref{section:mainthm} we introduce a more general class of allowed subsets $\fT \subset J^r\cE$ of the jet bundle as well as give a sharper range of degrees. We also take advantage of that section to give the definition of jet ampleness and jet bundles in algebraic geometry.

\subsection{Motivations and applications}

Motivated by their stabilisation result in the Grothendieck ring of varieties \cite{vakil_discriminants_2015}, Vakil and Wood conjectured that for a very ample line bundle $\cL$ on a smooth projective complex variety, the space of non-singular sections of $\cL^{\otimes d}$ should exhibit cohomological stability. In the special case of the projective space, Tommasi obtained the following result.
\begin{theorem}[Tommasi, \cite{tommasi_stable_2014}]
Let $d,n \geq 1$ be integers. Let $U_{d,n}$ be the space of non-singular holomorphic sections of $\cO(d)$ on $\bC\bP^n$. The rational cohomology of $U_{d,n}$ is isomorphic to the rational cohomology of the space $\mathrm{GL}_{n+1}(\bC)$ in degrees $* < \frac{d+1}{2}$.
\end{theorem}
\noindent In work in progress, she furthermore investigates an extension of this result to arbitrary smooth projective varieties \cite{tommasi_stable_nodate}. Using different techniques, O. Banerjee also confirmed the conjecture of Vakil and Wood in the case of smooth projective curves~\cite{banerjee_filtration_2021}.

The present work was strongly motivated by the result of Tommasi and the wish to understand the stable cohomology from a more homotopy theoretic point of view. At the time of writing, let us in particular mention the following result:
\begin{theorem}[Tommasi, work in progress in \cite{tommasi_stable_nodate}]
Let $X$ be a smooth projective complex variety of dimension $n$ and $\cL$ be a very ample line bundle on $X$. Let $d \geq 1$ be an integer and $U_d$ be the space of non-singular holomorphic sections of $\cL^{\otimes d}$. There is a Vassiliev spectral sequence converging to the homology of $U_d$. Working with rational coefficients, this spectral sequence degenerates on the $E_2$-page in the stable range if and only if the stable cohomology is a free commutative graded algebra on the cohomology of $X$ shifted by one degree.

Assuming this degeneration, the rational cohomology of $U_d$ in degrees $* < \lfloor\frac{d+1}{2}\rfloor$ is given by the free commutative graded algebra $\Lambda\left(H^{*-1}(X; \bQ) \right)$ on the cohomology of $X$ shifted by one degree.
\end{theorem}

In the last section (Section~\ref{section:app}) of this paper, we apply our main theorem to spaces of smooth hypersurfaces to prove a homological stability result with rational coefficients.
\begin{theorem}[see Theorem~\ref{thm:rationalcohomology}]
Let $X$ be a smooth projective complex variety and $\cL$ be a very ample line bundle on $X$. The rational cohomology ring of the space $\Gamma_{\mathrm{hol,ns}}(\cL^d)$ of non-singular sections (in the classical sense) of the $d$-th tensor power of $\cL$ is isomorphic to $\Lambda\left(H^{*-1}(X; \bQ) \right)$ in degrees $* < \frac{d-1}{2}$.
\end{theorem}

Firstly, let us point out that this agrees with the work in progress of Tommasi. In fact, one can use our main theorem to show the degeneration of the Vassiliev spectral sequence she constructed. Secondly, in contrast to many other instances of homological stability, one should remark that there are no natural stabilisation maps from spaces of non-singular sections of $\cL^d$ to those of $\cL^{d+1}$. Thus, we only mean that the cohomology rings abstractly stabilise, and the answer only depends on $X$ and not on $\cL$. After the apparition of the first version of the present article, and using different tools, Das and Howe proved a version of the above theorem for hypersurfaces in algebraic varieties over any algebraically closed field \cite{das_cohomological_2022}.

\bigskip

On the other hand, we find it quite interesting to notice that there is in general no integral homological stability. In fact, we prove the following result about the moduli space of smooth hypersurfaces of degree $d$ in $\bC\bP^2$:
\begin{theorem}[see Proposition~\ref{prop:zmodtwocomputation}]
Let $d \geq 6$ be an integer. We have:
\[
    H_2(\Gamma_{\mathrm{hol,ns}}(\bC\bP^2, \cO(d)); \bZ/2) \cong \begin{cases} \bZ/2 & d \equiv 0 \mod 2 \\ 0 & d \equiv 1 \mod 2. \end{cases}
\]
\end{theorem}
Besides the phenomenon this result illustrates, its proof showcases the potential of homotopical methods allowed by our main theorem. Indeed, the computation comes down to simple manipulations of Steenrod squares where the parity of $d$ is reflected in the Chern class of $\cO(d)$. In contrast, a more classical approach following the work of Vassiliev \cite{vassiliev_how_1999} would require knowledge of non-trivial differentials in spectral sequences that quickly grow out of hand when $d$ increases.

For good measure, we also study the $p$-torsion in the homology of $\Gamma_{\mathrm{hol,ns}}(\cL^d)$ and show that it stabilises when $p \geq \dim_\bC X + 2$ and $d \to \infty$. (See Proposition~\ref{prop:stableptorsion}.)

\bigskip

The results of this paper are also inspired by analogies with theorems in arithmetic probabilities, such as Poonen's Bertini theorem over finite fields \cite{poonen_bertini_2004}, and in motivic statistics in the Grothendieck ring of varieties as in \cite{vakil_discriminants_2015} or \cite{bilu_motivic_2019}. The recent results of Bilu and Howe particularly influenced the current formulation of our main theorem and we would like to recommend the introduction of their paper \cite{bilu_motivic_2019} to the reader interested in an overview of these analogies. Finally, we also wish to mention that I. Banerjee recently announced a result relating non-singular sections of a line bundle on an algebraic curve and smooth sections of the same line bundle \cite{banerjee_stable_2020}. 

\subsection{Acknowledgements}
I would like to thank my PhD advisor Søren Galatius for suggesting to compare algebraic sections to continuous sections of jet bundles. It is a pleasure to thank him for his encouragement and many helpful discussions. I would also like to thank Orsola Tommasi for discussing and sharing her work with me, as well as Ronno Das for helpful discussions related to this project. I was supported by the Danish National Research Foundation through the Copenhagen Centre for Geometry and Topology (DNRF151) as well as the European Research Council (ERC) under the European Union's Horizon 2020 research and innovation programme (grant agreement No. 682922).

\section{Statement of the main theorem}
\label{section:mainthm}

We begin by a few preliminary definitions before stating precisely our main theorem. Throughout this article, $X$ is a smooth projective complex variety and $\cE$ is a holomorphic vector bundle on $X$. We denote by $\Gamma$ the space of sections of a vector bundle, and decorate it with subscripts ``$\mathrm{hol}$'' or $\cC^0$ to indicate respectively holomorphic or continuous sections. We will make extensive use of \v{C}ech (or sheaf, as they will agree in our setting) cohomology with compact support which we denote by $\CCH^*$ and refer to \cite{bredon_sheaf_1997} for its definition and standard properties. All homology and cohomology groups will be taken with integral coefficients, unless otherwise specified. We recall the following definition of jet ampleness.

\begin{definition}[Compare \cite{beltrametti_generation_1999}]
Let $k \geq 0$ be an integer. Let $x_1, \ldots, x_t$ be $t$ distinct points in $X$ and $(k_1,\ldots,k_t)$ be a $t$-uple of positive integers such that $\sum_i k_i = k+1$. Denote by $\cO$ the structure sheaf of $X$ and by $\fm_i$ the maximal ideal sheaf corresponding to $x_i$. We regard the tensor product $\otimes_{i=1}^t \fm_i^{k_i}$ as a subsheaf of $\cO$ under the multiplication map $\otimes_{i=1}^t \fm_i^{k_i} \to \cO$. We say that $\cE$ is \emph{$k$-jet ample} if the evaluation map
\[
    \Gammahol\left(\cE\right) \lra \Gammahol\left(\cE \otimes \left(\cO / \otimes_{i=1}^t \fm_i^{k_i}\right)\right) \cong \bigoplus_{i=1}^t \Gammahol\left(\cE \otimes \left(\cO / \fm_i^{k_i}\right)\right)
\]
is surjective for any $x_1,\ldots,x_t$ and $k_1,\ldots,k_t$ as above.
\end{definition}

\begin{example}\label{example:jetampleness}
A vector bundle $\cE$ is $0$-jet ample if and only if it is spanned by its global sections. In the case of a line bundle, $1$-jet ampleness corresponds to the usual notion of very ampleness. On a curve, a line bundle is $k$-jet ample whenever it is $k$-very ample. However, on higher dimensional varieties, a $k$-jet ample line bundle is also $k$-very ample but the converse is not true in general. Finally, and most importantly for us, if $\cA$ and $\cB$ are holomorphic vector bundles which are respectively $a$- and $b$-jet ample, then their tensor product $\cA \otimes \cB$ is $(a+b)$-jet ample. (See~\cite[Proposition 2.3]{beltrametti_generation_1999}.)
\end{example}

To ease the readability of various statements throughout the paper, we will use the following notation.

\begin{definition}\label{def:bigN}
For a holomorphic vector bundle $\cE$ on $X$ and an integer $r \in \bN$, we define $N(\cE,r) \geq 0$ to be the largest integer $N$ such that $\cE$ is $\left( (N+1) \cdot (r+1) - 1\right)$-jet ample. If no such integer exists, we set $N(\cE,r) = - 1$, although we shall never consider such a case in this paper.
\end{definition}

Let us also recall the construction of the jet bundle from~\cite[IV.16.7]{grothendieck_elements_1967} (where it is called the sheaf of principal parts). The diagonal morphism $\Delta \colon X \to X \times X$ gives a surjection of sheaves $\Delta^\sharp \colon \Delta^*\cO_{X \times X} \to \cO_X$. Denoting by $\cI$ the kernel, we have $\cO_X \cong \Delta^* \cO_{X \times X}/\cI$. For an integer $r \geq 0$, we define the $r$-th jet bundle of $\cO_X$ to be
\[
    J^r\cO_X := \Delta^* \cO_{X \times X} / \cI^{r+1}.
\]
The projections $p_i \colon X \times X \to X$ give two $\cO_X$-algebra structures on $J^r\cO_X$ and, unless otherwise specified, we use the one given by the first projection $p_1$. The other morphism induced by $p_2$ is denoted by
\[
    d^r_X \colon \cO_X \lra J^r\cO_X.
\]
For a holomorphic vector bundle $\cE$ on $X$, we define its \emph{$r$-th jet bundle} to be
\begin{equation}
    J^r \cE := J^r\cO_X \otimes_{\cO_X} \cE
\end{equation}
where $J^r\cO_X$ is seen as an $\cO_X$-module via the morphism $d^r_X$ for the tensor product, and the result is regarded as an $\cO_X$-module again via $p_1$. It comes with the morphism
\[
    d^r_{X,\cE} := d^r_X \otimes \cE \colon \cE \lra J^r\cO_X \otimes_{\cO_X} \cE = J^r\cE.
\]
Taking global sections, we obtain the \emph{jet map}:
\begin{align}\label{eqn:jetmap}
    j^r = \Gamma(d^r_{X,\cE}) \colon \Gammahol(\cE) \lra \Gammahol(J^r\cE).
\end{align}
The most important observation for us is the following: if $x \in X$ is a point with maximal ideal sheaf $\fm$, the fibre $(J^r\cE)|_x$ is naturally identified with the complex vector space $\cE_x / \fm_x^{r+1} \cE_x$. Furthermore, the composition
\[
    \cE_x \overset{(d^r_{X,\cE})_x}{\lra} (J^r\cE)_x \lra (J^r\cE)|_x = \cE_x / \fm_x^{r+1} \cE_x
\]
is the natural quotient map. (Here, and everywhere else, we write $\cE_x$ for the stalk of the sheaf $\cE$ and $\cE|_x = \cE_x / \fm_x \cE_x$ for the fibre of the bundle $\cE$.) Intuitively, for a holomorphic section $s$ of $\cE$, one should think of the value of $j^r(s)$ at a point $x\in X$ as the tuple of all derivatives of $s$ at $x$ up to order $r$. In particular, the following lemma is a direct consequence of the definitions.
\begin{lemma}\label{lemma:multijet}
Let $\cE$ be a holomorphic vector bundle on X and let $N(\cE,r)$ be as in Definition~\ref{def:bigN}. Let $(x_0, \ldots, x_p)$ be a tuple of $p+1$ distinct points in $X$. If $p \leq N(\cE,r)$, the simultaneous evaluation of the jet map~\eqref{eqn:jetmap} at these points
\begin{align*}
    j^r_{(x_0,\ldots,x_p)} \colon \Gammahol(\cE) &\lra (J^r\cE)|_{x_0} \times \cdots \times (J^r\cE)|_{x_p} \\
                                    s &\longmapsto \left(j^r(s)(x_0), \ldots, j^r(s)(x_p)\right)
\end{align*}
is surjective.\qed
\end{lemma}

We shall now explain what we precisely mean by restricting the behaviour of sections of $\cE$. In particular, we will require certain subsets of the jet bundle to be ``semi-algebraic''. This is a technical condition which is quite arbitrary. We believe that a clearer and more general notion could be used, but we were unfortunately not able to make the arguments of Section~\ref{section:cohomologygroups} work without it. Our arguments rely on multiple properties of these sets: they admit cell decompositions, have a well-defined dimension, and they behave well under projections and closure. (See Section~\ref{subsection:laststepfiltration} for their single but crucial use.)

\bigskip

There is a well-studied concept of real semi-algebraic subsets of an Euclidean space. They are subsets defined by polynomial equations and inequalities.

\begin{definition}[Compare \cite{bochnak_real_1998}]\label{def:realsemialgebraicsubset}
A semi-algebraic subset of $\bR^n$ is a union of finitely many subsets of the form
\[
    \left\{ x\in \bR^n \mid P(x) = 0, \ Q_1(x) > 0, \ldots, \ Q_l(x) > 0 \right\},
\]
where $l \in \bN$ and $P,Q_1,\ldots,Q_l \in \bR[X_1,\ldots,X_n]$.
\end{definition}

We adapt the definition to families, i.e. to subsets of vector bundles, by demanding the standard definition to be satisfied locally in charts. This is well-defined because an algebraic variety $X$ has an atlas whose transition functions are algebraic, hence respect the semi-algebraicity.

Let us be more precise. First, we briefly recall the notion of an algebraic atlas on $X$. To lighten the notation, we let $n$ be the complex dimension of $X$ and $m$ be the complex rank of $J^r\cE$. We denote by $V(-)$ the vanishing set of the tuple of polynomials.

The variety $X$ can be covered by Zariski open subsets, each of the form
\[
    U \cong V(f_1,\ldots,f_{d-n}) \subset \bC^d
\]
for some integer $d \geq 1$ and polynomials $f_1, \ldots, f_{d-n}$. Furthermore, if $U$ and $W$ are Zariski open subsets of $X$ with $\alpha \colon U \cong V(f_1,\ldots,f_{d-n}) \subset \bC^d$ and $\beta \colon W \cong V(g_1,\ldots,g_{d'-n}) \subset \bC^{d'}$, the homeomorphism on the intersection
\[
    \alpha(W \cap U) \cap V(f_1,\ldots,f_{d-n}) \overset{\cong}{\lra} W \cap U \overset{\cong}{\lra} \beta(U \cap W) \cap V(g_1,\ldots,g_{d'-n})
\]
is given by a rational function whose domain is a subset of $\bC^d$ and codomain is a subset of $\bC^{d'}$. Recall also that the algebraic vector bundle $J^r\cE$ is equivalently given by the data of trivialising Zariski open subsets $U_i \subset X$ (over which $J^r\cE|_{U_i} \cong U_i \times \bC^m$) and transition functions on overlaps $U_i \cap U_j \to \mathrm{GL}_m(\bC)$. Most importantly for us, the transition functions are regular morphisms.

\begin{definition}\label{def:semialgebraic}
Let $n$ be the complex dimension of $X$ and $m$ be the complex rank of $J^r\cE$. A subset $\fT \subset J^r\cE$ is \emph{real semi-algebraic} if there exists a cover $X = \bigcup U_i$ by Zariski open subsets such that the following conditions hold for each $i$:
\begin{enumerate}
    \item the jet bundle may be trivialised over $U_i$ via a map $\varphi_i \colon J^r\cE|_{U_i} \overset{\cong}{\to} U_i \times \bC^m$;
    \item there is a chart $\phi_i \colon U_i \overset{\cong}{\to} V(f^i_1,\ldots,f^i_{d_i-n}) \subset \bC^{d_i}$ for some polynomials $f^i_1,\ldots,f^i_{d_i-n}$;
    \item and the image in $\bR^{2(d_i + m)}$ of $\fT|_{U_i}$ via the map
    \[
        J^r\cE|_{U_i} \overset{\varphi_i}{\lra} U_i \times \bC^m \overset{\phi_i \times \mathrm{id}}{\lra} V(f^i_1,\ldots,f^i_{d_i-n}) \times \bC^m \subset \bC^{d_i + m} \cong \bR^{2(d_i + m)}
    \]
    is a real semi-algebraic subset. (Here $\fT|_{U_i}$ is the restriction of $\fT$ above $U_i$.)
\end{enumerate}
\end{definition}
We will often drop the adjective ``real'' as we will never consider any complex analogue. In essence, a subset $\fT \subset J^r\cE$ is semi-algebraic in the sense of Definition~\ref{def:semialgebraic} when it is semi-algebraic in the usual way when ``read in charts''. As all the change-of-coordinates maps described above are rational functions, being semi-algebraic is independent of the choice of the cover. Indeed, the image of a semi-algebraic set by a rational function is still semi-algebraic (see \cite[Section 2.2]{bochnak_real_1998}). 

A semi-algebraic subset has a well-defined dimension (as in \cite[Section 2.8]{bochnak_real_1998}) which can be thought of as the maximal dimension in a decomposition into cells of the form $]0,1[^d$ (see \cite[Corollary 2.8.9]{bochnak_real_1998}). We therefore get a well-defined dimension for a semi-algebraic subset $\fT \subset J^r\cE$ by looking at the dimensions when ``reading in charts'':
\begin{definition}
Let $\fT \subset J^r\cE$ be a semi-algebraic subset. Let $X = \bigcup U_i$ by a finite cover as in Definition~\ref{def:semialgebraic} (the finiteness can always be arranged by compactness of $X$) and write $\fT_{U_i} \subset \bR^{2(d_i + m)}$ for the semi-algebraic sets obtained using the condition (3). Each of them has a well-defined dimension and we let the \emph{dimension of} $\fT$ be their maximum.
\end{definition}

In the following definition, we denote by $\rank_\bC J^r\cE$ the complex rank of $J^r\cE$.

\begin{definition}\label{def:admissibletaylor}
We say that a subset $\fT \subset J^r\cE$ is an \emph{admissible Taylor condition} if it is closed, real semi-algebraic and has dimension at most $2(\rank_\bC J^r\cE - 1)$. We will use the notation $\fT|_x := (J^r\cE)|_x \cap \fT$ for the fibre above a point $x \in X$.
\end{definition}

\begin{remark}
Although our definition is quite technical and general, the typical admissible Taylor conditions arise as subvarieties of high enough codimension. Indeed, any closed subvariety $\fT \subset J^r\cE$ of the jet bundle of complex codimension at least $\dim_\bC X + 1$ defines an admissible Taylor condition.
\end{remark}

Motivated by the previous remark, and to help general bookkeeping throughout the paper, we will use the following notation.
\begin{definition}\label{def:excesscodimension}
The (real) \emph{excess codimension} of an admissible Taylor condition $\fT$ is defined to be the number $e(\fT) = \mathrm{codim}_\bR \fT - 2\dim_\bC X \geq 2$, where $\mathrm{codim}_\bR \fT$ is the real codimension of $\fT$ in the jet bundle $J^r\cE$.
\end{definition}

We are now ready to define what it means for a section to be singular with respect to an admissible Taylor condition $\fT$.

\begin{definition}\label{def:singular}
A holomorphic section $s$ of the vector bundle $\cE$ is said to be \emph{singular} if there exists a point $x \in X$ such that $j^r(s)(x) \in \fT|_x$.
Similarly, a (continuous) section $s$ of the vector bundle $J^r\cE$ is said to be \emph{singular} if there exists a point $x \in X$ such that $s(x) \in \fT|_x$. A section that is not singular is said to be \emph{non-singular}.
\end{definition}

\begin{example}
If $\cE$ is a line bundle, we may take $\fT$ to be the zero section of $J^1\cE$. It is an admissible Taylor condition and a singular section is one that vanishes at a point on $X$ where its derivative also vanishes. In particular, if $s$ is a non-singular section, its zero set $Z(s) := \{x \in X \mid s(x)=0 \} \subset X$ is a smooth submanifold.
\end{example}

When talking about spaces of sections $\Gamma$, we will use the subscript ``$\mathrm{ns}$" to denote the subspace of \emph{\textbf{n}on-\textbf{s}ingular} sections. The following is our main result.
\begin{theorem}\label{thm:mainthm}
Let $r \geq 0$ and $N \geq 1$ be integers. Let $\cE$ be an $\left( (N+1) \cdot (r+1) - 1\right)$-jet ample vector bundle on $X$ and let $\fT \subset J^r\cE$ be an admissible Taylor condition. The composition
\[
    \Gamma_{\mathrm{hol,ns}}\left( \cE \right) \overset{j^r}{\lra} \Gamma_{\mathrm{hol,ns}}\left( J^r\cE \right) \hookrightarrow \Gamma_{\cC^0, \mathrm{ns}}\left( J^r\cE \right)
\]
induces an isomorphism in homology:
\[
    H_*\left(\Gamma_{\mathrm{hol, ns}}(\cE); \bZ\right) \lra H_*\left(\Gamma_{\cC^0, \mathrm{ns}}(J^r\cE); \bZ\right)
\]
in the range of degrees $* < N(\cE,r)\cdot (e(\fT)-1) + e(\fT) - 2$.
\end{theorem}

\subsection{Outline of the paper}

There are two key ingredients in the proof of the main theorem~\ref{thm:mainthm}. The first one is a spectral sequence à la Vassiliev (see \cite[Chapter III]{vassiliev_complements_1994} for an analogous statement in the case of smooth sections, and \cite{vassiliev_how_1999} for more explicit computations in small degrees). The starting idea is that one should study the space of singular sections and deduce the homology of the space of non-singular sections via Alexander duality. The former has a natural filtration given by counting the number of singularities and it is used to construct a spectral sequence converging to its cohomology. Comparing spectral sequences allows us to compare sections of $\cE$ and $J^r\cE$. The second ingredient is a version of the classical Stone--Weierstrass theorem adapted from the work of Mostovoy \cite{mostovoy_spaces_2006} which allows us to compare holomorphic and continuous sections of $J^r\cE$.

We first explain how to resolve the singular subspaces and construct the Vassiliev spectral sequence in Section~\ref{section:resolution}. We study its first page in Section~\ref{section:cohomologygroups}. In Section~\ref{section:interpolation}, we explain how to go from holomorphic sections to continuous sections. Then, in Section~\ref{section:compareSS}, we construct a morphism of spectral sequences and use it to compare various spaces of sections. We finish proving our main theorem in Section~\ref{section:holcontinuouscompare}. Lastly, in Section~\ref{section:app}, we apply our results to study spaces of non-singular sections of a very ample line bundle on a projective variety.

\section{Resolution of singularities}\label{section:resolution}

In this section, we choose an admissible Taylor condition $\fT \subset J^r\cE$ inside the $r$-th jet bundle of a holomorphic vector bundle $\cE$ on $X$, and we will write for brevity
\[
    \Gamma = \Gammahol\left(\cE\right) \quad \text{and} \quad \Sigma = \Gammahol\left(\cE\right) - \Gamma_{\mathrm{hol, ns}}\left(\cE\right)
\]
for the vector space $\Gamma$ of all holomorphic sections of $\cE$ and its subspace $\Sigma$ of singular sections. We also define the \emph{singular space} of a section $f \in \Gamma$
\begin{align}\label{eqn:defsingularset}
    \mathrm{Sing}(f) := \left\{ x \in X \mid j^r(f)(x) \in \fT \right\} \subset X
\end{align}
as the space of points where $f$ is singular (as in Definition~\ref{def:singular}). Our final goal, Theorem~\ref{thm:mainthm}, is to understand the homology of the space of non-singular sections $\Gamma_{\mathrm{hol, ns}}\left(\cE\right) = \Gamma - \Sigma$. By Alexander duality
\[
    \CCH^{i}(\Sigma) \cong \widetilde{H}_{2\dim_\bC \Gamma - i -1}(\Gamma - \Sigma),
\]
it is equivalent to understand the compactly supported \v{C}ech cohomology of its complement $\Sigma$. To achieve that, we want to construct a spectral sequence converging to $\CCH^*(\Sigma)$. This spectral sequence arises from a resolution of the space $\Sigma$ which we define in this section.

\subsection{Construction of the resolution}

We will construct a space $R\fX \to \Sigma$ mapping surjectively to the singular subspace $\Sigma$. The inverse image of a section $f\in \Sigma$ with $j+1$ singularities will be a $j$-simplex $\Delta^j$. This will allow us to show that $R\fX \to \Sigma$ induces an isomorphism in cohomology with compact supports (up to some modifications). The space $R\fX$ will be advantageously filtered by subspaces $R^j\fX$ related via pushout diagrams resembling the skeletal decomposition of a simplicial space. This filtration then yields a spectral sequence computing the cohomology of $R\fX$, hence that of $\Sigma$.

This is inspired by the so-called truncated resolution of Mostovoy~\cite{mostovy_truncated_2012} but written in a more functorial way as in~\cite{vokrinek_generalization_2007}.

\bigskip

In what follows, the space $\Gamma$ is given its canonical topology coming from the fact that it is a finite dimensional complex vector space. Let $\catF$ be the category whose objects are the finite sets $[n] := \{0,\ldots,n\}$ for $n \geq 0$ and whose morphisms are \emph{all} maps of sets $[n] \to [m]$. Let $\catTop$ be the category of topological spaces and continuous maps between them. We define the following functor
\begin{equation}\label{eqn:defoffunctorX}
\begin{split}
    \fX \colon \catF^\op &\lra \catTop \\
                [n] &\longmapsto \fX[n] := \{(f,s_0,\ldots,s_n) \in \Gamma \times X^{n+1} \mid \forall i, \ s_i \in \mathrm{Sing}(f) \}
\end{split}
\end{equation}
where $\fX[n]$ is given the subspace topology from $\Gamma \times X^{n+1}$. On morphisms, for a map of sets $g \colon [n] \to [m]$, we define 
\begin{equation*}
\begin{split}
    \fX(g) \colon \fX[m] &\lra \fX[n] \\
            (f,s_0,\ldots,s_m) &\longmapsto (f,s_{g(0)}, \ldots, s_{g(n)}).
\end{split}
\end{equation*}
For an integer $k \geq 0$, we denote by $\catF_{\leq k}$ the full sub-category of $\catF$ on objects $[n]$ for $n \leq k$. Let us also write 
\[
    |\Delta^n| = \{ (t_0, \ldots, t_n) \mid \forall i, 0 \leq t_i \leq 1 \text{ and } t_0 + \cdots + t_n = 1 \} \subset \bR^{n+1}
\]
for the standard topological $n$-simplex, and denote by $\partial |\Delta^n|$ its boundary. In particular, the assignment $[n] \mapsto |\Delta^n|$ gives a functor $\catF \to \catTop$. For an integer $j \geq 0$, we define the \emph{$j$-th geometric realisation of $\fX$} by the following coend:
\begin{equation}\label{eqn:jthgeometricrealisationcoend}
\begin{split}
    R^j\fX :=& \int^{[n] \in \catF_{\leq j}} \fX[n] \times |\Delta^n| \\
            =& \left( \bigsqcup_{0 \leq n \leq j} \fX[n] \times |\Delta^n| \right) / \sim
\end{split}
\end{equation}
where the equivalence relation $\sim$ is generated by $(\fX(g)(z),t) \sim (z,g_*(t))$ for all maps $g \colon [n] \to [m]$ in $\catF$. (Here $g_* \colon |\Delta^n| \to |\Delta^m|$ denotes the usual map induced on the simplices by functoriality.) This is of course reminiscent of the classical geometric realisation of a simplicial space. Note however that here a cell $|\Delta^n|$ in the geometric realisation is indexed by an \emph{unordered} set of singularities, even though the functor $\fX$ is defined using ordered tuples. Indeed, all the permutations $[n] \to [n]$ are valid morphisms in our category $\catF$.

\bigskip

Let $j \geq 1$ be an integer. We now describe how $R^j\fX$ may be obtained from $R^{j-1}\fX$ via a pushout diagram. Let $L_j$ be the following set:
\begin{equation}\label{eqn:deflatchingobject}
    L_j := \{ (f,s_0, \ldots, s_j) \in \Gamma \times X^{j+1} \mid \exists l \neq k \text{ such that } s_l = s_k \} \subset \fX[j]
\end{equation}
topologised as a subspace of $\fX[j]$. This should be thought of as the analogue of the ``latching object'' of a simplicial space. We denote by 
\[
    L_j \times_{\fS_{j+1}} |\Delta^j|
\]
the quotient space of $L_j \times |\Delta^j|$ by the symmetric group $\fS_{j+1}$ acting on $L_j$ by permuting the singularities $s_i$, and on $|\Delta^j|$ by permuting the coordinates. Denote by $\widehat{\cdot}$ the omission of an element in a tuple.
\begin{lemma}\label{lemma:smallwelldefinedmap}
The formula
\[
    \left( (f,s_0,\ldots,s_j), (t_0, \ldots, t_j) \right) \longmapsto 
    \begin{cases} \left( (f,s_0,\ldots,\widehat{s_l},\ldots,s_j), (t_0, \ldots, t_k + t_l, \ldots, \widehat{t_l}, \ldots, t_j) \right) \\
     \text{ if there exists $k \neq l$ such that $s_l = s_k$ } \end{cases}
\]
gives a well-defined map $L_j \times_{\fS_{j+1}} |\Delta^j| \to R^{j-1}\fX$.
\end{lemma}
\begin{proof}
The formula appears ill-defined as we are choosing arbitrarily two indices $k$ and $l$. The identifications made by the coend formula~\eqref{eqn:jthgeometricrealisationcoend} show that any choice will yield the same class in the quotient.
\end{proof}
Recall that a point $t = (t_0, \ldots, t_j) \in |\Delta^j|$ is in the boundary $\partial |\Delta^j|$ if one of its coordinates vanishes. An argument similar to the proof of Lemma~\ref{lemma:smallwelldefinedmap} above gives the following.
\begin{lemma}\label{lemma:secondsmallwelldefinedmap}
The formula
\[
    \left( (f,s_0,\ldots,s_j), (t_0, \ldots, t_j) \right) \longmapsto 
    \left( (f,s_0,\ldots,\widehat{s_l},\ldots,s_j), (t_0, \ldots, \widehat{t_l}, \ldots, t_j) \right)
     \text{ if $t_l = 0$ }
\]
gives a well-defined map $\fX[j] \times_{\fS_{j+1}} \partial |\Delta^j| \to R^{j-1}\fX$. \qed
\end{lemma}

Consider the following pushout diagram of spaces:
\begin{center}
\begin{tikzcd}
L_j \times_{\fS_{j+1}} \partial|\Delta^j| \arrow[d, hook] \arrow[r, hook]  \arrow[dr, phantom, "\ulcorner", very near end] & { \fX[j] \times_{\fS_{j+1}} \partial |\Delta^j|} \arrow[d]                                                           \\
L_j \times_{\fS_{j+1}} |\Delta^j| \arrow[r]                               & {\left( L_j \times_{\fS_{j+1}} |\Delta^j|\right) \bigcup \left(\fX[j] \times_{\fS_{j+1}} \partial |\Delta^j|\right)}.
\end{tikzcd}
\end{center}
Equivalently, the pushout is the union of the top-right and bottom-left spaces inside $\fX[j] \times_{\fS_{j+1}} |\Delta^j|$.
The maps defined above in Lemma~\ref{lemma:smallwelldefinedmap} and Lemma~\ref{lemma:secondsmallwelldefinedmap} glue to a continuous map
\[
    \alpha_{j-1} \colon \left( L_j \times_{\fS_{j+1}} |\Delta^j|\right) \bigcup \left(\fX[j] \times_{\fS_{j+1}} \partial |\Delta^j|\right) \lra R^{j-1}\fX.
\]
The natural map $\fX[j] \times |\Delta^j| \to R^j\fX$ factors through the quotient by the symmetric group action and gives a map
\[
    \beta_j \colon \fX[j] \times_{\fS_{j+1}} |\Delta^j| \lra R^j\fX.
\]
From the coend formula~\eqref{eqn:jthgeometricrealisationcoend} and the inclusion of the full sub-category $\catF_{\leq j-1} \subset \catF_{\leq j}$, we also get a natural map $R^{j-1}\fX \to R^j\fX$. We are now ready to state the

\begin{proposition}
The following square is a pushout diagram of topological spaces:
\begin{equation}\label{eqn:resolutionpushoutdiagram}
\begin{tikzcd}
{\left( L_j \times_{\fS_{j+1}} |\Delta^j|\right) \bigcup \left(\fX[j] \times_{\fS_{j+1}} \partial |\Delta^j|\right) } \arrow[d, hook] \arrow[r, "\alpha_{j-1}"] \arrow[dr, phantom, "\ulcorner", very near end] & R^{j-1}\fX \arrow[d] \\
{\fX[j] \times_{\fS_{j+1}} |\Delta^j|} \arrow[r, "\beta_j"']                                                                        & R^j\fX.             
\end{tikzcd}
\end{equation}
\end{proposition}
\begin{proof}
We may construct the pushout $P$ as the quotient
\[
    P := \left( R^{j-1}\fX \bigsqcup \fX[j] \times_{\fS_{j+1}} |\Delta^j| \right) / \sim.
\]
One may check that the map $\beta_j$ together with the natural map $R^{j-1}\fX \to R^j\fX$ gives a map from the disjoint union above which factors through the quotient. Hence we get a well-defined map $P \to R^j\fX$. We now construct a continuous inverse. Recall that $R^j\fX$ is defined in~\eqref{eqn:jthgeometricrealisationcoend} as a quotient of
\[
    \left( \bigsqcup_{0 \leq n \leq j-1} \fX[n] \times |\Delta^n| \right) \bigsqcup \left( \fX[j] \times |\Delta^j| \right).
\]
The natural map $\left( \bigsqcup_{0 \leq n \leq j-1} \fX[n] \times |\Delta^n| \right) \to R^{j-1}\fX \to P$ together with the identity of $\fX[j] \times |\Delta^j|$ gives a map from the disjoint union that factors through the quotient and yields a well-defined map $R^j\fX \to P$. One may finally verify that it is the inverse of the map $P \to R^j\fX$ constructed above.
\end{proof}

We now turn to proving some topological results about our constructions.
\begin{lemma}\label{lemma:Xnisclosed}
For any integer $n \geq 0$, the subspace $\fX[n] \subset \Gamma \times X^{n+1}$ defined in~\eqref{eqn:defoffunctorX} is closed.
\end{lemma}
\begin{proof}
Let $\mathrm{ev} \colon \Gamma \times X^{n+1} \to (J^r\cE)^{n+1}$ be the simultaneous evaluation of the jet map $j^r$ (defined in~\eqref{eqn:jetmap}) at $(n+1)$ points of $X$. We observe directly from the definitions that $\fX[n] = \mathrm{ev}^{-1}(\fT^{n+1})$, hence is closed as the inverse image of a closed set.
\end{proof}

\begin{lemma}\label{lemma:rhonisproper}
For any $n \geq 0$, the map $\rho_n \colon \fX[n] \to \Gamma$ given by $(f,s_0,\ldots,s_n) \mapsto f$ is a proper map.
\end{lemma}
\begin{proof}
The projection onto the first factor $\Gamma \times X^{n+1} \to \Gamma$ is proper as $X^{n+1}$ is compact. Hence so is its restriction $\rho_n$ to the closed subspace $\fX[n]$.
\end{proof}

In particular, the map $\rho_n$ is closed, so $\Sigma = \rho_1(\fX[1])$ is closed in $\Gamma$. We have natural projections maps $\fX[n] \times |\Delta^n| \to \fX[n] \overset{\rho_n}{\to} \Gamma$ for any $n \geq 0$. They give rise to a map
\begin{equation}\label{eqn:defoftauj}
    \tau_j \colon R^j\fX \lra \Sigma
\end{equation}
for every integer $j \geq 0$. 
\begin{lemma}\label{lemma:taujprojectionisproper}
For any integer $j \geq 0$, the map $\tau_j \colon R^j\fX \to \Sigma$ is a proper map.
\end{lemma}
\begin{proof}
We have to show that the preimage of any compact set is compact. Equivalently, because $\Sigma$ is locally compact and Hausdorff, we will show that $\tau_j$ is a closed map with compact fibres.
From Lemma~\ref{lemma:rhonisproper}, for any $n$, the map $\rho_n$ is closed and hence so is the composition $\fX[n] \times |\Delta^n| \to \fX[n] \overset{\rho_n}{\to} \Gamma$. This implies that $\tau_j$ is closed. It remains to see that it has compact fibres. If $f \in \Sigma$, we observe that $\tau_j^{-1}(f) = \beta_j \left( \rho_j^{-1}(f) \right)$ which is compact as $\rho_j^{-1}(f)$ is, by Lemma~\ref{lemma:rhonisproper}.
\end{proof}

A major advantage of the pushout square~\eqref{eqn:resolutionpushoutdiagram} is that it allows us to prove the following topological lemma.
\begin{lemma}\label{lemma:paracompactHausdorff}
For any integer $j \geq 0$, the space $R^j\fX$ is paracompact and Hausdorff. Furthermore, the natural map $R^{j-1}\fX \to R^j\fX$ is a closed embedding.
\end{lemma}
\begin{proof}
Firstly, from Lemma~\ref{lemma:Xnisclosed}, we know that $R^0\fX = \fX[0] \subset \Gamma \times X$ is a closed subset, hence is itself paracompact Hausdorff. Then the lemma is proven inductively using the pushout diagram~\eqref{eqn:resolutionpushoutdiagram} together with the fact that 
\[
    \left( \left( L_j \times_{\fS_{j+1}} |\Delta^j|\right) \bigcup \left(\fX[j] \times_{\fS_{j+1}} \partial |\Delta^j|\right) \right) \hookrightarrow \fX[j] \times_{\fS_{j+1}} |\Delta^j|
\]
is a closed embedding. 
\end{proof}
In the sequel, using the closed embedding of Lemma~\ref{lemma:paracompactHausdorff} just above, we will simply write $R^{j-1}\fX \subset R^j\fX$. For an integer $j \geq 0$, we let 
\begin{equation}\label{eqn:Yjsubspacepairwisedistinct}
    Y_j := \left\{ (f,s_0, \ldots, s_j) \in \fX[j] \mid s_l \neq s_k \text{ if } l \neq k \right\} = \fX[j] - L_j \subset \fX[j]
\end{equation}
be the subspace of $\fX[j]$ where the singularities are pairwise distinct. For later use, we record the following homeomorphism, which is a direct consequence of the pushout square~\eqref{eqn:resolutionpushoutdiagram} and the fact that the vertical maps therein are closed embeddings:
\begin{equation}\label{eqn:filtrationleafhomeo}
    R^j\fX - R^{j-1}\fX \cong Y_j \times_{\fS_{j+1}} \mathrm{Interior}(|\Delta^j|).
\end{equation}

\bigskip

Let us now discuss why $\tau_j \colon R^j\fX \to \Sigma$ needs to be slightly modified to obtain a meaningful ``resolution'' of $\Sigma$. The fibre $\tau_j^{-1}(f)$ above a section $f \in \Sigma$ that has at most $j+1$ singularities is by construction a $j$-simplex. Hence it is contractible and one might hope that $\tau_j$ induces an isomorphism in cohomology. This is unfortunately not the case. Indeed, $\tau_j^{-1}(f)$ is not contractible if $f$ has at least $j+2$ singularities. To fix this problem, we will modify $R^j(\Sigma)$ by gluing a cone over each fibre $\tau_j^{-1}(f)$ which is not contractible. The precise construction is as follows.

\bigskip

Let $N \geq 0$ be an integer. We let
\begin{equation}
    \Sigma_{\geq N+2} := \left\{ f \in \Gamma \mid \# \mathrm{Sing}(f) \geq N+2 \right\} \subset \Sigma
\end{equation}
denote the subspace of those sections with at least $N+2$ singularities. We denote by $\overline{\Sigma_{\geq N+2}}$ its closure in $\Sigma$ (or equivalently, in $\Gamma$). Observe that the surjectivity of the map $\tau_N$ implies the following equality:
\[
    \tau_N \left( \tau_N^{-1}\left(\overline{\Sigma_{\geq N+2}}\right)  \right) = \overline{\Sigma_{\geq N+2}}.
\]
We glue fibrewise a cone over each $f \in \overline{\Sigma_{\geq N+2}}$ by defining the space $R^N_{\text{cone}}(\Sigma)$ as the following homotopy pushout:
\begin{equation}\label{eqn:homotopypushoutresolution}
\begin{tikzcd}
\tau_N^{-1}\left(\overline{\Sigma_{\geq N+2}}\right) \arrow[d, "\tau_N"'] \arrow[r, hook] \arrow[dr, phantom, "^\mathrm{ho}\ulcorner", very near end] & R^N\fX \arrow[d]      \\
\overline{\Sigma_{\geq N+2}} \arrow[r]                                           & R^N_{\text{cone}}\fX.
\end{tikzcd}
\end{equation}
All three defining spaces in the corners of~\eqref{eqn:homotopypushoutresolution} map to $\Sigma$, hence we obtain a surjective projection map
\begin{equation}\label{eqn:piprojectionconed}
    \pi \colon R^N_{\text{cone}}\fX \lra \Sigma.
\end{equation}
We want to prove that $\pi$ induces an isomorphism in \v{Cech} cohomology with compact supports. We begin with a couple of lemmas.

\begin{lemma}\label{lemma:piisproper}
The map $\pi \colon R^N_{\text{cone}}\fX \lra \Sigma$ is proper.
\end{lemma}
\begin{proof}
We will prove that is it closed with compact fibres, which implies the properness. By definition of the homotopy pushout, $R^N_{\text{cone}}\fX$ is a quotient of the following disjoint union:
\[
    R^N\fX \ \bigsqcup \ \tau_N^{-1}\left(\overline{\Sigma_{\geq N+2}}\right) \times [0,1] \ \bigsqcup \ \overline{\Sigma_{\geq N+2}}.
\]
The map $\pi$ is induced by the following three maps: the projection $\tau_N \colon R^N\fX \to \Sigma$, the projection $\tau_N^{-1}\left(\overline{\Sigma_{\geq N+2}}\right) \times [0,1] \to \tau_N^{-1}\left(\overline{\Sigma_{\geq N+2}}\right) \to \Sigma$, and the inclusion $\overline{\Sigma_{\geq N+2}} \hookrightarrow \Sigma$. The first two are closed by Lemma~\ref{lemma:taujprojectionisproper} and the last one is the inclusion of a closed subset, hence closed.

Finally, we prove that the fibres of $\pi$ are compact. We saw in the proof of Lemma~\ref{lemma:taujprojectionisproper} that for any $f \in \Sigma$, the fibre $\tau_N^{-1}(f)$ was compact. Now, $\pi^{-1}(f)$ is either $\tau_N^{-1}(f)$ if $f \in \Sigma - \overline{\Sigma_{\geq N+2}}$ or a cone over it if $f \in \overline{\Sigma_{\geq N+2}}$. In any case it is compact.
\end{proof}

\begin{lemma}
The space $R^N_{\text{cone}}\fX$ is paracompact, locally compact, and Hausdorff.
\end{lemma}
\begin{proof}
The paracompactness and Hausdorffness follow from the definition as a homotopy pushout and Lemma~\ref{lemma:paracompactHausdorff}. It is locally compact as its maps properly to the locally compact space $\Sigma$.
\end{proof}

These topological properties will justify our subsequent manipulations of compactly supported \v{C}ech cohomology, which agrees with sheaf cohomology with compact supports in this context. The most important corollary is the following
\begin{proposition}\label{prop:piinducesiso}
The map $\pi \colon R^N_{\text{cone}}\fX \lra \Sigma$ induces an isomorphism in \v{C}ech cohomology with compact supports.
\end{proposition}
\begin{proof}
The properness of $\pi$ proved in Lemma~\ref{lemma:piisproper} implies that it induces a well-defined map in cohomology with compact supports. We also observed in the proof of that lemma that a fibre of $\pi$ is either a simplex or a cone, hence contractible. The proposition then follows from the Vietoris--Begle theorem~\cite[V.6.1]{bredon_sheaf_1997}.
\end{proof}

\subsection{Construction of the spectral sequence}

Let $N \geq 1$ be an integer. Recall from Lemma~\ref{lemma:paracompactHausdorff} that we have closed embeddings $R^{j-1}\fX \subset R^j\fX$. We define the following filtration on $R^N_{\text{cone}}\fX$:
\[
    F_0 = R^0\fX \subset F_1 = R^1\fX \subset \cdots \subset F_N = R^N\fX \subset F_{N+1} = R^N_{\text{cone}}\fX.
\]
Following standard arguments, we obtain from the filtration a spectral sequence:
\[
    E_1^{p,q} = \CCH^{p+q}(F_p, F_{p-1}) \cong \CCH^{p+q}(F_p - F_{p-1}) \implies \CCH^{p+q}(R^N_{\text{cone}}\fX)
\]
where the isomorphism between the cohomology groups on the first page follows from \cite[II.12.3]{bredon_sheaf_1997}. Using Proposition~\ref{prop:piinducesiso} and Alexander duality, we obtain:
\[
    \CCH^{p+q}(R^N_{\text{cone}}\fX) \cong \CCH^{p+q}(\Sigma) \cong \widetilde{H}_{2\dim_\bC \Gamma - (p+q) -1}(\Gamma - \Sigma)
\]
where $\widetilde{H}$ denotes reduced singular homology. Letting $s = -p -1$ and $t = 2\dim_\bC \Gamma - q$, we regrade our spectral sequence and obtain the following
\begin{proposition}\label{prop:thereexistsSS}
There is a spectral sequence on the second quadrant $s \leq -1$ and $t \geq 0$:
\[
    E^1_{s,t} = \CCH^{2\dim_\bC \Gamma - 1 - s - t}(F_{-s-1} - F_{-s-2}; \bZ) \implies \widetilde{H}_{s+t}(\Gamma - \Sigma; \bZ).
\]
The differential $d^r$ on the $r$-th page of the spectral sequence has bi-degree $(-r,r-1)$, i.e. it is a morphism $d^r_{s,t} \colon E^r_{s,t} \to E^r_{s-r,t+r-1}$.
\end{proposition}

\section{Cohomology groups on the $E^1$-page}\label{section:cohomologygroups}

As in the last section, we choose a holomorphic vector bundle $\cE$ on $X$ and an admissible Taylor condition $\fT \subset J^r\cE$ inside the $r$-th jet bundle of $\cE$. For the remainder of this section, we also let 
\[
    N = N(\cE,r)
\]
be the largest integer $N \geq 0$ such that $\cE$ is $\left( (N+1) \cdot (r+1) - 1\right)$-jet ample as in Definition~\ref{def:bigN}. As discussed in the introduction, we assume that such an $N$ exists. If not, the statements in this section are either trivially false, or trivially true as they describe elements of the empty set. For brevity, we still use the following notations
\[
    \Gamma = \Gammahol\left(\cE\right) \quad \text{and} \quad \Sigma = \Gammahol\left(\cE\right) - \Gamma_{\mathrm{hol, ns}}\left(\cE\right)
\]
as well as $\fX$ for the associated functor $\catF^\op \to \catTop$ as in~\eqref{eqn:defoffunctorX}.

We will study the first page of the spectral sequence from Proposition~\ref{prop:thereexistsSS} converging to the cohomology of $R^N_{\text{cone}}\fX$:
\[
    E^1_{s,t} = \CCH^{2\dim_\bC \Gamma - 1 - s - t}(F_{-s-1} - F_{-s-2}; \bZ).
\]
We will first show that for $-N-1 \leq s \leq -1$ the groups $E^1_{s,t}$ can be written, via Thom isomorphisms, in terms of the cohomology of $\fT$. We will then study qualitatively the cohomology of $F_{N+1} - F_{N}$, i.e. the column $E^1_{-N-2,*}$, and show that it does not have any influence on the cohomology of the limit in a range of degrees up to around $N$. Later in Section~\ref{section:interpolation} we will construct spectral sequences for spaces of sections of $J^1\cE \setminus \fT$, and in Section~\ref{section:compareSS} we will compare them. The explicit computations of the present section will show that these various spectral sequences are isomorphic in a range from the first page and onwards.

\subsection{The first steps of the filtration}

For an integer $j \geq 0$, recall from~\eqref{eqn:Yjsubspacepairwisedistinct} the space 
\[
    Y_j = \left\{ (f,s_0, \ldots, s_j) \in \fX[j] \mid s_l \neq s_k \text{ if } l \neq k \right\} \subset \fX[j].
\]
\begin{lemma}\label{lemma:firstbundle}
For $0 \leq j \leq N(\cE,r)$, there is a fibre bundle:
\[
    \mathrm{Interior}\left(|\Delta^j|\right) \lra F_j - F_{j-1} \lra Y_j / \fS_{j+1}.
\]
\end{lemma}
\begin{proof}
Recall from the definition of the filtration on $R^N_{\text{cone}}\fX$ that $F_j = R^j\fX$ for $0 \leq j \leq N$. As a consequence of the pushout square~\eqref{eqn:resolutionpushoutdiagram}, we observed in~\eqref{eqn:filtrationleafhomeo} that we have the following homeomorphism:
\[
    R^j\fX - R^{j-1}\fX \cong Y_j \times_{\fS_{j+1}} \mathrm{Interior}(|\Delta^j|).
\]
Projecting down to the first factor gives the required fibre bundle.
\end{proof}

By an \emph{affine bundle} we mean a torsor for a vector bundle. In the sequel, they will arise naturally from fibrewise surjective linear maps between vector bundles. For any integer $j \geq 1$, the bundle $(J^r\cE)^j$ projects down to $X^j$ and we may consider its restriction to the open subset $\Conf_j(X) \subset X^j$ of those tuples of points which are pairwise distinct. The symmetric group $\fS_j$ acts on these spaces by permuting the coordinates. In particular, it acts on the subspace $\fT^j \subset (J^r\cE)^j$ and we let
\begin{equation}\label{eqn:TaylorconditionaboveConfdivided}
    \fT^{(j)} := \left( \fT^j|_{\Conf_j(X)} \right) / \fS_j
\end{equation}
be the orbit space of the restriction of $\fT^j$ over the subspace $\Conf_j(X) \subset X^j$.

\begin{lemma}\label{lemma:secondbundle}
Let $0 \leq j \leq N(\cE,r)$ be an integer and recall from~\eqref{eqn:Yjsubspacepairwisedistinct} the space $Y_j$ of those tuples $(f,s_0,\ldots,s_j) \in \Gamma \times \Conf_{j+1}(X)$ where $f$ is singular at the $s_i$. We may simultaneously evaluate the jet map at these points:
\begin{align*}
    Y_j &\lra \fT^{j+1}|_{\Conf_{j+1}(X)} \\
    (f,s_0,\ldots,s_j) &\longmapsto (j^r(f)(s_0), \ldots, j^r(f)(s_j)).
\end{align*}
Taking $\fS_{j+1}$-orbits on the domain and codomain of this map yields an affine bundle:
\[
    Y_j / \fS_{j+1} \lra \fT^{(j+1)}
\]
whose fibre has complex dimension $\dim_\bC \Gamma - (j+1) \rank_\bC J^r\cE$. (Here $\rank_\bC J^r\cE$ denotes the complex rank of the vector bundle $J^r\cE$.)
\end{lemma}
\begin{proof}
The simultaneous evaluation of the jet map gives a map
\begin{equation}\label{eqn:secondlemmasurjectionvectbundle}
\begin{tikzcd}
\Gamma \times \Conf_{j+1}(X) \arrow[rd] \arrow[rr] &                & (J^r\cE)^{j+1}|_{\Conf_{j+1}(X)} \arrow[ld] \\
                                                   & \Conf_{j+1}(X) &                                            
\end{tikzcd}
\end{equation}
of vector bundles over the configuration space $\Conf_{j+1}(X)$. Under the assumption $0 \leq j \leq N(\cE,r)$, Lemma~\ref{lemma:multijet} shows that this map of bundles is fibrewise surjective. Therefore the top map of~\eqref{eqn:secondlemmasurjectionvectbundle} is an affine bundle. Subtracting the ranks, we obtain that its fibre has complex dimension $\dim_\bC \Gamma - (j+1)\rank_\bC J^r\cE$.

Now, restricting the domain of~\eqref{eqn:secondlemmasurjectionvectbundle} to $Y_j$ and its codomain to $\fT^{j+1}|_{\Conf_{j+1}(X)}$, and then taking $\fS_{j+1}$-orbits yields the following affine bundle:
\[
    Y_j / \fS_{j+1} \lra \left( \fT^{j+1}|_{\Conf_{j+1}(X)} \right) / \fS_{j+1} = \fT^{(j+1)}
\]
which still has the rank that we have computed above.
\end{proof}

The quotient maps $Y_j \to Y_j / \fS_{j+1}$ and $\fT^{j+1}|_{\Conf_{j+1}(X)} \to \fT^{(j+1)}$ are principal $\fS_{j+1}$-bundles and hence are classified by (homotopy classes of) maps to the classifying space $B\fS_{j+1}$. Composing with the sign representation $B\fS_{j+1} \overset{B\mathrm{sign}}{\lra} B\bZ / 2$, we obtain two well-defined homotopy classes of maps:
\[
    Y_j / \fS_{j+1} \lra B\bZ /2 \quad \text{ and } \quad \fT^{(j+1)} \lra B\bZ /2.
\]
We will write $\bZ^\mathrm{sign}$ for the corresponding local coefficient systems.

\begin{proposition}\label{prop:firstpagethomiso}
Let $-N(\cE,r)-1 \leq s \leq -1$. Then, we have the following isomorphism:
\[
    E^1_{s,t} \cong \CCH^{-t - 2s \cdot \rank_\bC J^r\cE}(\fT^{(-s)}; \bZ^\mathrm{sign})
\]
where $\fT^{(-s)}$ is the space defined in~\eqref{eqn:TaylorconditionaboveConfdivided} and $\bZ^\mathrm{sign}$ is the local coefficient system described above.
\end{proposition}

\begin{proof}
Recall from Proposition~\ref{prop:thereexistsSS} that the first page of the spectral sequence is given by
\[
    E^1_{s,t} = \CCH^{2\dim_\bC \Gamma - 1-s-t}(R^{-s-1}\fX - R^{-s-2}\fX; \bZ).
\]
Via a homeomorphism $\mathrm{Interior}\left(|\Delta^j|\right) \cong \bR^j$, we see that the fibre bundle of Lemma~\ref{lemma:firstbundle} is homeomorphic to a vector bundle. Applying the Thom isomorphism to the latter, we obtain:
\[
    E^1_{s,t} \cong \CCH^{2\dim_\bC \Gamma - t}(Y_{-s-1} / \fS_{-s}; \bZ^\mathrm{sign}).
\]
Another application of the Thom isomorphism using Lemma~\ref{lemma:secondbundle} yields
\[
    E^1_{s,t} \cong \CCH^{-t - 2s \cdot \rank_\bC J^r\cE}(\fT^{(-s)}; \bZ^\mathrm{sign}).\qedhere
\]
\end{proof}

\subsection{The last step of the filtration}\label{subsection:laststepfiltration}

We study the last non-trivial part of the $E^1$-page, that is the column $s = -N(\cE,r)-2$ where:
\[
    E^1_{-N-2,t} = \CCH^{2\dim_\bC \Gamma + 1+N - t}(R^N_\text{cone}\fX - R^N\fX; \bZ).
\]
The methods from the last section do not apply to the space $R^N_\text{cone}\fX - R^N\fX$ and we will not be able to express the cohomology groups $E^1_{-N-2,t}$ in terms of other ``known'' groups. However, using the technical assumptions made in Definition~\ref{def:admissibletaylor} about the Taylor condition $\fT$, we will obtain a vanishing result for $E^1_{-N-2,t}$. This will be enough for the proof of our main theorem.

\bigskip

Recall the projection map $\tau_N \colon R^N\fX \to \Sigma$ from~\eqref{eqn:defoftauj}. From the homotopy pushout square~\eqref{eqn:homotopypushoutresolution}, we obtain the following homeomorphism:
\[
    R^N_\text{cone}\fX - R^N\fX \cong \left( \left( \tau_N^{-1}\left(\overline{\Sigma_{\geq N+2}}\right) \times ]0,1] \right) \bigsqcup \overline{\Sigma_{\geq N+2}} \right) / \sim
\]
where $(z,1) \in \tau_N^{-1}\left(\overline{\Sigma_{\geq N+2}}\right) \times ]0,1]$ is identified with $\tau_N(z) \in \overline{\Sigma_{\geq N+2}}$ in the quotient. Indeed, there is a natural continuous bijection from the right-hand side to the left-hand side. It is in fact a homeomorphism, as the top arrow in the homotopy pushout square~\eqref{eqn:homotopypushoutresolution} is the inclusion of a closed subset. In other words, this is the fibrewise (for the map $\tau_N$) open cone over $\overline{\Sigma_{\geq N+2}}$. We stratify this space by the following locally closed subspaces (this is analogous to~\cite[Lemma 18]{tommasi_stable_2014}):
\begin{align*}
    \Str_{-1} &:= \overline{\Sigma_{\geq N+2}}, \\
    \Str_0 &:= \left( \tau_N^{-1}\left(\overline{\Sigma_{\geq N+2}}\right) \times ]0,1[ \right)  \cap  \left( R^0\fX \times ]0,1[\right), \\
    \Str_j &:= \left( \tau_N^{-1}\left(\overline{\Sigma_{\geq N+2}}\right) \times ]0,1[ \right)  \cap  \left( (R^j\fX - R^{j-1}\fX) \times ]0,1[\right) \quad \text{ for } 1 \leq j \leq N.
\end{align*}
For $0 \leq j \leq N$, let
\begin{equation}
    Y_j^{\geq N+2} := \left\{ (f,s_0,\ldots,s_j) \in \Gamma \times \Conf_{j+1}(X) \mid f \in \overline{\Sigma_{\geq N+2}} \text{ and } s_i \in \mathrm{Sing}(f) \right\} \subset Y_j.
\end{equation}
Using the homeomorphism~\eqref{eqn:filtrationleafhomeo} identifying the difference between two consecutive steps of the resolution, we have a homeomorphism
\begin{equation}\label{eqn:filtrationfjoncone}
    \Str_j \cong \left( Y_j^{\geq N+2} \times_{\fS_{j+1}} \mathring{|\Delta^j|} \right) \times ]0,1[.
\end{equation}
for $0 \leq j \leq N$, where $\mathring{|\Delta^j|}$ denotes the interior of the simplex.

It is easier to think about this stratification by looking at one fibre $\pi^{-1}(f)$ at a time. Then, we are just decomposing an open cone over a union of simplices into the following pieces: the apex (corresponding to $\Str_{-1} \cap \pi^{-1}(f)$), the open segments from the 0-simplices to the apex (corresponding to $\Str_0 \cap \pi^{-1}(f)$), the open (filled) triangles between the 1-simplices and the apex, etc. Figure~\ref{fig:strata} below shows the strata in a single fibre $\pi^{-1}(f)$ when $f$ has $3$ singular points and $N = 1$. In this case, $\tau_N^{-1}(f)$ consists of three $1$-simplices glued together (i.e. a triangle), so $\pi^{-1}(f)$ is the cone over that triangle.
\def\OFFSETA{6}
\def\OFFSETB{12}
\def\OFFSETC{18}
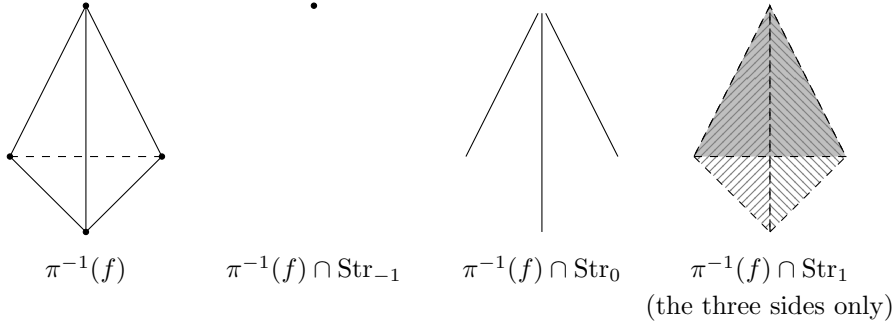
\begin{figure}[H]
\centering
\begin{tikzpicture}[scale=0.5]
\filldraw (1,2) circle (2pt);
\filldraw (3,0) circle (2pt);
\filldraw (5,2) circle (2pt);
\filldraw (3,6) circle (2pt);

\draw (1,2) -- (3,0);
\draw (3,0) -- (5,2);
\draw (3,0) -- (3,6);
\draw (1,2) -- (3,6);
\draw (5,2) -- (3,6);
\draw[dashed] (1,2) -- (5,2);

\node at (3,-1) {$\pi^{-1}(f)$};

\filldraw (3+\OFFSETA,6) circle (2pt);
\node at (3+\OFFSETA,-1) {$\pi^{-1}(f) \cap \Str_{-1}$};

\draw (1+\OFFSETB, 2) -- (2.9+\OFFSETB, 5.8);
\draw (3+\OFFSETB, 0) -- (3+\OFFSETB, 5.8);
\draw (5+\OFFSETB, 2) -- (3.1+\OFFSETB, 5.8);
\node at (3+\OFFSETB,-1) {$\pi^{-1}(f) \cap \Str_0$};

\fill[lightgray] (1+\OFFSETC,2) -- (5+\OFFSETC,2) -- (3+\OFFSETC,6);
\fill[pattern=north east lines, pattern color=gray] (1+\OFFSETC,2) -- (3+\OFFSETC,0) -- (3+\OFFSETC,6);
\fill[pattern=north west lines, pattern color=gray] (3+\OFFSETC,0) -- (5+\OFFSETC,2) -- (3+\OFFSETC,6);
\draw[dashed] (1+\OFFSETC,2) -- (3+\OFFSETC,0) -- (3+\OFFSETC,6) -- cycle;
\draw[dashed] (3+\OFFSETC,0) -- (5+\OFFSETC,2) -- (3+\OFFSETC,6) -- cycle;
\draw[dashed] (1+\OFFSETC,2) -- (5+\OFFSETC,2) -- (3+\OFFSETC,6) -- cycle;
\node at (3+\OFFSETC,-1) {$\pi^{-1}(f) \cap \Str_1$};
\node at (3+\OFFSETC,-2) {(the three sides only)};

\end{tikzpicture}
\caption{Decomposition of the open cone.} \label{fig:strata}
\end{figure}

\bigskip

If we find an integer $D \geq 0$ such that $\CCH^k(\Str_j) = 0$ for all $-1 \leq j \leq N$ and all $k > D$, then the same result will hold for the union, i.e. $\CCH^k(R^N_\text{cone}\fX - R^N\fX) = 0$ for $k > D$. In what follows, we set out to find such a $D$ as small as we can. With that in mind, we make the following ad hoc definition of cohomological dimension: 
\begin{definition}
We say that a space $Z$ has \emph{cohomological dimension} $D$ with respect to a local coefficient system $\cA$ if $D$ is the smallest integer such that $\CCH^k(Z; \cA) = 0$ for all $k > D$. We will denote it by $\cohodim(Z, \cA)$, or simply $\cohodim(Z)$ if $\cA = \bZ$.
\end{definition}

The only non-trivial local coefficient system we will need is $\bZ^\mathrm{sign}$, which is induced on the quotient $Y_j^{\geq N+2} / \fS_{j+1}$ by the sign representation $\fS_{j+1} \to \bZ/2$.
\begin{lemma}\label{lemma:firstestimationFilj}
For $0 \leq j \leq N$, we have 
\[
    \cohodim(\Str_j) = 1+j+ \cohodim\left(Y_j^{\geq N+2} / \fS_{j+1}, \bZ^\mathrm{sign}\right).
\]
\end{lemma}
\begin{proof}
From the homeomorphism~\eqref{eqn:filtrationfjoncone}, we have a trivial fibre bundle
\[
    ]0,1[ \lra \Str_j \lra Y_j^{\geq N+2} \times_{\fS_{j+1}} \mathring{|\Delta^j|}.
\]
This implies that $\cohodim(\Str_j) = 1 + \cohodim\left(Y_j^{\geq N+2} \times_{\fS_{j+1}} \mathring{|\Delta^j|}\right)$.
Now, we have another fibre bundle:
\[
    \mathring{|\Delta^j|} \lra Y_j^{\geq N+2} \times_{\fS_{j+1}} \mathring{|\Delta^j|} \lra Y_j^{\geq N+2} / \fS_{j+1}.
\]
Hence, by the Thom isomorphism, we obtain:
\[
    \cohodim\left(Y_j^{\geq N+2} \times_{\fS_{j+1}} \mathring{|\Delta^j|}\right) = j + \cohodim\left(Y_j^{\geq N+2} / \fS_{j+1}, \bZ^\mathrm{sign}\right).\qedhere
\]
\end{proof}

We thus have reduced our problem to studying the cohomology of $Y_j^{\geq N+2} / \fS_{j+1}$ for $0 \leq j \leq N$, as well as that of $\overline{\Sigma_{\geq N+2}}$. We shall do so by comparing these spaces to a known one, namely the space
\[
    Y_N = \left\{ (f,s_0, \ldots, s_N) \in \Gamma \times \Conf_{N+1}(X) \mid s_i \in \mathrm{Sing}(f) \right\}.
\]
First, let us introduce some notation. Using charts on $X$, we may cover $Y_N$ by finitely many semi-algebraic sets, whose intersections are also semi-algebraic. Recall, e.g. from~\cite[Theorem 2.3.6]{bochnak_real_1998}, that every semi-algebraic set is the disjoint union of cells, each homeomorphic to an open disc $]0,1[^d$ for some $d \geq 0$. The largest $d$ in such a decomposition is called the dimension of the semi-algebraic set. Let $\dim Y_N$ be the largest of the dimensions of the semi-algebraic sets in a cover of $Y_N$. (It depends a priori of the chosen cover, but we suppress this from the notation.) The following lemma is a crucial result for controlling our spectral sequence.

\begin{lemma}\label{lemma:crucialestimationlemma}
For $0 \leq j \leq N$, we have
\[
    \dim Y_N \geq \cohodim\left(Y_j^{\geq N+2} / \fS_{j+1}, \bZ^\mathrm{sign} \right).
\]
\end{lemma}
\begin{proof}
Let $0 \leq j \leq N$. Only remembering the $(j+1)$\textsuperscript{st} singularities gives a map
\begin{equation}\label{eqn:semialgebraicsurjectivemap}
\begin{split}
    Y_N &\lra Y_j^{\geq N+2} \\
    (f,s_0,\ldots,s_N) &\longmapsto (f, s_0,\ldots,s_j).
\end{split}
\end{equation}
Notice that this map is not surjective, as it may happen that a section $f \in \overline{\Sigma_{\geq N+2}}$ has fewer than $N+1$ singularities. We study the map~\eqref{eqn:semialgebraicsurjectivemap} locally via charts. Let $U_0, \ldots, U_N \subset X$ be charts on $X$ as in Definition~\ref{def:semialgebraic}. Then the subsets 
\[
    U := \left\{ (f,s_0,\ldots,s_j) \in \overline{\Sigma_{\geq N+2}} \times U_0 \times \cdots \times U_j \mid s_k \in \mathrm{Sing}(f), s_i \neq s_j \forall i\neq j \right\} \subset Y_j^{\geq N+2}
\]
and
\[
    V := \left\{ (f,s_0,\ldots,s_N) \in \Gamma \times U_0 \times \cdots \times U_N \mid s_k \in \mathrm{Sing}(f), s_i \neq s_j \forall i\neq j \right\} \subset Y_N
\]
are semi-algebraic. Indeed, they are the preimages of the semi-algebraic set $\fT^{j+1}$ (respectively $\fT^{N+1}$) via the simultaneous evaluation of the jet map which is algebraic, hence semi-algebraic. (See~\cite[Proposition 2.2.7]{bochnak_real_1998}.) The restriction of the map~\eqref{eqn:semialgebraicsurjectivemap} to $U$ and $V$ is an algebraic map, hence semi-algebraic map, $\phi \colon U \to V$ between semi-algebraic sets. Using \cite[Theorem 2.8.8]{bochnak_real_1998}, we obtain the following inequality on the dimensions (as defined above using cell decompositions):
\[
    \dim(U) \geq \dim(\phi(U)).
\]
Furthermore, the definition of $Y_j^{\geq N+2}$ implies that the semi-algebraic map $\phi \colon U \to V$ has dense image, i.e. $\overline{\phi(U)} = V$. Using that the closure has the same dimension (\cite[Proposition 2.8.2]{bochnak_real_1998}) and the inequality above, we obtain:
\[
    \dim(U) \geq \dim(V).
\]
Varying the charts $U_0,\ldots,U_N \subset X$, we may cover the domain and codomain of~\eqref{eqn:semialgebraicsurjectivemap} by subsets defined like $U$ and $V$. If $U'$ and $V'$ are two other such subsets, then $U \cap U'$ and $V \cap V'$ are also semi-algebraic sets because they are intersections of semi-algebraic sets. (This follows from the Definition~\ref{def:realsemialgebraicsubset}.) Hence the argument shows that the inequality on the dimensions holds also on intersections. Let $\dim Y_j^{\geq N+2}$ denote the maximum of the dimensions in a cover of $Y_j^{\geq N+2}$ by semi-algebraic sets. Then, an argument using the Mayer--Vietoris spectral sequence shows that the cohomological dimension of $Y_j^{\geq N+2}$ is less than its dimension $\dim Y_j^{\geq N+2}$. Therefore
\begin{equation}\label{eqn:cohodimYnYjinequality}
    \dim Y_N \geq \dim Y_j^{\geq N+2} \geq \cohodim\left(Y_j^{\geq N+2}\right).
\end{equation}
Finally, from the principal $\fS_{j+1}$-bundle $Y_j^{\geq N+2} \to Y_j^{\geq N+2} / \fS_{j+1}$, we see that the dimension of the orbit space is the same as that of $Y_j^{\geq N+2}$. Therefore the inequality~\eqref{eqn:cohodimYnYjinequality} holds when replacing the rightmost term with $\cohodim(Y_j^{\geq N+2} / \fS_{j+1}, \bZ^\mathrm{sign})$.
\end{proof}

\noindent Repeating the proof with the map $Y_N \to \overline{\Sigma_{\geq N+2}}$, $(f,s_0,\ldots,s_N) \mapsto f$ yields the
\begin{lemma}\label{lemma:estimationcohodimapexofcone}
The following inequality holds:
\[
\pushQED{\qed} 
    \dim Y_N \geq \cohodim(\overline{\Sigma_{\geq N+2}}, \bZ). \qedhere
\popQED
\]
\end{lemma}

\noindent The final computation to be made is the content of the following lemma. It uses the notation $e(\fT)$ of excess codimension established in Definition~\ref{def:excesscodimension}.
\begin{lemma}\label{lemma:YNcohodimestimation}
The dimension of $Y_N$ satisfies:
\[
    \dim Y_N \leq 2 \dim_\bC \Gamma - (N+1)e(\fT).
\]
\end{lemma}
\begin{proof}
The proof of Lemma~\ref{lemma:secondbundle} shows that the simultaneous evaluation of the jet map
\begin{align*}
    Y_N &\lra \fT^{N+1}|_{\Conf_{N+1}(X)} \\
    (f,s_0,\ldots,s_N) &\longmapsto (j^r(f)(s_0), \ldots, j^r(f)(s_N))
\end{align*}
is an affine bundle whose fibre has complex dimension $\dim_\bC \Gamma- (N+1)\rank_\bC J^r\cE$. Therefore, on dimensions:
\[
    \dim Y_N \leq \dim(\fT^{N+1}|_{\Conf_{N+1}(X)}) + 2\dim_\bC \Gamma - 2(N+1)\rank_\bC J^r\cE.
\]
Now, because $\fT$ is a semi-algebraic subset of $J^r\cE$ of dimension less than $2\rank_\bC J^r\cE - e(\fT)$, we obtain that:
\[
    \dim(\fT^{N+1}|_{\Conf_{N+1}(X)}) \leq (N+1)(2\rank_\bC J^r\cE - e(\fT)).
\]
The lemma is then proven by combining these two inequalities.
\end{proof}

Assembling all the estimations we have obtained so far, we can state and prove the following.
\begin{proposition}\label{prop:lastcolumnvanishingrange}
The cohomology groups in the column $s = - N(\cE,r) - 2$ on the first page of the spectral sequence:
\[
    E^1_{-N-2,t} = \CCH^{2\dim_\bC \Gamma + 1+N - t}(R^N_\text{cone}\fX - R^N\fX; \bZ)
\]
vanish for $t < N\cdot e(\fT) + e(\fT)$.
\end{proposition}
\begin{proof}
We had set up the stratification $\Str_j$, $-1 \leq j \leq N$, on $R^N_\text{cone}\fX - R^N\fX$ so that
\[
    \cohodim\left( R^N_\text{cone}\fX - R^N\fX \right) \leq \max_{j} \cohodim\left( \Str_j \right).
\]
For $0 \leq j \leq N$, combining Lemma~\ref{lemma:firstestimationFilj}, Lemma~\ref{lemma:crucialestimationlemma}, and Lemma~\ref{lemma:YNcohodimestimation}, we get:
\[
    \cohodim\left( \Str_j \right) \leq 1+j + 2\dim_\bC \Gamma -(N+1)e(\fT)  \leq 2\dim_\bC \Gamma - N(e(\fT)-1) - (e(\fT)-1).
\]
Similarly, using Lemma~\ref{lemma:estimationcohodimapexofcone} and Lemma~\ref{lemma:YNcohodimestimation}, we obtain:
\[
    \cohodim\left( \Str_{-1} \right) \leq 2 \dim_\bC \Gamma - (N+1) e(\fT).
\]
Therefore $\cohodim\left( R^N_\text{cone}\fX - R^N\fX \right) \leq 2\dim_\bC \Gamma - N(e(\fT)-1) - (e(\fT)-1)$ and the result follows.
\end{proof}

\subsection{Differentials and summary}\label{subsection:differentialsummary}

Summing up all the results so far, we have the following proposition. 
\begin{proposition}\label{prop:summingUPresolutionandSS}
Let $\cE$ be a holomorphic vector bundle on $X$ and $\fT \subset J^r\cE$ be an admissible Taylor condition. Let $N = N(\cE,r)$. The resolution and its filtration described in Section~\ref{section:resolution} give rise to a spectral sequence on the second quadrant $s \leq -1$ and $t \geq 0$ converging to the homology of the space of non-singular sections $\Gamma_{\mathrm{hol, ns}}(\cE)$:
\[
    E^1_{s,t} = \CCH^{2\dim_\bC \Gamma - 1 - s - t}(F_{-s-1} - F_{-s-2}; \bZ) \implies \widetilde{H}_{s+t}(\Gamma_{\mathrm{hol, ns}}(\cE); \bZ).
\]
The differentials on the $r$-th page have bi-degree $(-r,r-1)$. Furthermore, for $-N-1 \leq s \leq -1$, we have the following isomorphisms for all $t \geq 0$:
\[
    E^1_{s,t} \cong \CCH^{-t - 2s \cdot \rank_\bC J^r\cE}(\fT^{(-s)}; \bZ^\mathrm{sign}).
\]
(The space $\fT^{(-s)}$ is defined in~\eqref{eqn:TaylorconditionaboveConfdivided}.) Moreover, for $t < N\cdot e(\fT) + e(\fT)$:
\[
    E^1_{-N-2,t} = 0.
\]
\end{proposition}

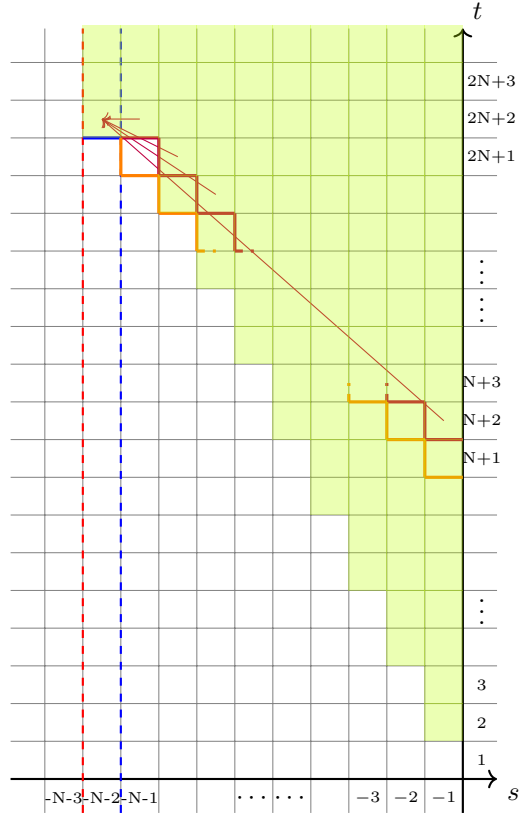
\begin{figure}[H]
\centering
\begin{tikzpicture}[scale=0.5]
\draw[step=1cm,gray,ultra thin] (-11.9,-.9) grid (.9,19.9);
\draw[thick,->] (-11.9,0) -- (.9,0) node[anchor=north west] {$s$};
\draw[thick,->] (0,-.9) -- (0,19.9) node[anchor=south west] {$t$};

\foreach \x in {-1,-2,-3}
    \draw (\x+.5,-.5) node {\tiny $\x$};
\draw (-4.5,-.5) node {$\cdots$};
\draw (-5.5,-.5) node {$\cdots$};
\draw (-8.5,-.5) node {\tiny -N-1};
\draw (-9.5,-.5) node {\tiny -N-2};
\draw (-10.5,-.5) node {\tiny -N-3};

\foreach \y in {1,2,3}
    \draw (.5, \y-.5) node {\tiny $\y$};
\draw (.5,4.65) node {$\vdots$};
\draw (.5,8.5) node {\tiny N+1};
\draw (.5,9.5) node {\tiny N+2};
\draw (.5,10.5) node {\tiny N+3};
\draw (.5,12.65) node {$\vdots$};
\draw (.5,13.65) node {$\vdots$};
\draw (.5, 16.5) node {\tiny \quad 2N+1};
\draw (.5, 17.5) node {\tiny \quad 2N+2};
\draw (.5, 18.5) node {\tiny \quad 2N+3};

\draw[red, thick, dashed] (-10,-.9) -- (-10,19.9);
\draw[blue, thick, dashed] (-9,-.9) -- (-9,19.9);

\draw[blue, very thick] (-10,17) -- (-9,17);

\draw[purple, very thick] (-9,17) -- (-8,17);
\draw[purple, very thick] (-8,17) -- (-8,16);
\draw[purple, very thick] (-8,16) -- (-7,16);
\draw[purple, very thick] (-7,16) -- (-7,15);
\draw[purple, very thick] (-7,15) -- (-6,15);
\draw[purple, very thick] (-6,15) -- (-6,14);
\draw[purple, very thick, dashed] (-6,14) -- (-5.5,14);
\draw[purple, very thick] (0,9) -- (-1,9);
\draw[purple, very thick] (-1,9) -- (-1,10);
\draw[purple, very thick] (-1,10) -- (-2,10);
\draw[purple, very thick, dashed] (-2,10) -- (-2,10.5);

\draw[purple, <-] (-9.5, 17.5) -- (-8.5, 17.5);
\draw[purple, <-] (-9.5, 17.5) -- (-7.5, 16.5);
\draw[purple, <-] (-9.5, 17.5) -- (-6.5, 15.5);
\draw[purple, <-] (-9.5, 17.5) -- (-.5, 9.5);

\draw[orange, very thick] (-9,17) -- (-9,16);
\draw[orange, very thick] (-9,16) -- (-8,16);
\draw[orange, very thick] (-8,16) -- (-8,15);
\draw[orange, very thick] (-8,15) -- (-7,15);
\draw[orange, very thick] (-7,15) -- (-7,14);
\draw[orange, very thick, dashed] (-7,14) -- (-6.5,14);
\draw[orange, very thick] (0,8) -- (-1,8);
\draw[orange, very thick] (-1,8) -- (-1,9);
\draw[orange, very thick] (-1,9) -- (-2,9);
\draw[orange, very thick] (-2,9) -- (-2,10);
\draw[orange, very thick] (-2,10) -- (-3,10);
\draw[orange, very thick, dashed] (-3,10) -- (-3,10.5);

\foreach \s in {1,...,9} {
    \foreach \t in {1,...,19} {
        \ifthenelse{\t > \numexpr 2*\s-2}{
            \ColouredRectangle{-\s}{\t}
        }{}
    }
}

\ColouredRectangle{-10}{17}
\ColouredRectangle{-10}{18}
\ColouredRectangle{-10}{19}
\end{tikzpicture}
\caption{First page of the spectral sequence when $e(\fT) = 2$.} \label{fig:firstpagefigure}
\end{figure}

We briefly describe the zones in Figure~\ref{fig:firstpagefigure}, where we have chosen to fix $e(\fT) = 2$ to lighten the notation. Firstly, the only possibly non-vanishing groups lie in the \textcolor{lime}{coloured squares}. All groups $E^r_{s,t}$ with $s \leq -N-3$ are zero as the filtration finishes after $N+1$ steps. According to Proposition~\ref{prop:lastcolumnvanishingrange}, the groups below the \textcolor{blue}{horizontal solid line} in the column $s = -N-2$ vanish. The differentials coming from the groups below the \textcolor{purple}{upper staircase} never hit groups in the column where $s = -N-2$ and $t \geq 2N+2$. Finally, the \textcolor{orange}{lower staircase} delimits the zone of total degree $* \leq N-1$. We have also drawn \textcolor{purple}{some differentials} $d^r$ to the group $E^r_{-N-2,2N+2}$ for $r=1,2,3$ and $N+1$.

\section{Interpolating holomorphic and continuous sections}\label{section:interpolation}

In this section, we introduce and study section spaces that lie in-between holomorphic and continuous sections of the jet bundle $J^r\cE$. They will be written as combinations of holomorphic and ``anti-holomorphic'' sections. We first explain how to take the complex conjugate of a holomorphic section. We then construct these spaces and finish by explaining how the resolution and the spectral sequence from the previous sections can be adapted to them.

\subsection{Complex conjugation of sections}

Using the fact that $X$ is projective, we choose once and for all a very ample holomorphic line bundle $\cL$ on it as well as a basis $z_0, \ldots, z_M$ of the complex vector space of holomorphic global sections $\Gammahol(\cL)$.

We denote by $\overline{\cL}$ the complex conjugate line bundle of $\cL$. It is obtained from the underlying real vector bundle of $\cL$ by having the complex numbers act by multiplication by their complex conjugates. We regard it as a smooth complex line bundle. We now define a complex conjugation operation $\cL \to \overline{\cL}$. Recall that the line bundle $\cL$ may be constructed as a quotient
\[
    \cL := \left(\bigsqcup_i U_i \times \bC \right) \Big/ (x,v_i) \sim (x,t_{ji}(v_i))
\]
from the data $\left( \{U_i\}_i, (t_{ij})_{i,j} \right)$ of trivialising open sets $U_i \subset X$ and transition functions $t_{ij} \colon U_i \cap U_j \to \mathrm{GL}_1(\bC) = \bC^*$ satisfying a cocycle condition. Similarly, $\overline{\cL}$ may be constructed via such a quotient by replacing the transition functions by their complex conjugates $\overline{t_{ij}}$. The formula
\begin{equation*}
\begin{split}
    \bigsqcup_i U_i \times \bC &\lra \bigsqcup_i U_i \times \bC \\
    (x,v) &\longmapsto (x,\overline{v})
\end{split}
\end{equation*}
then gives a well defined $\bR$-linear isomorphism $\cL \to \overline{\cL}$. On continuous global sections, we thus obtain an $\bR$-linear \emph{complex conjugation operation}:
\begin{equation}\label{eqn:complexconjsections}
    \overline{\cdot} \colon \Gamma_{\cC^0}(\cL) \lra \Gamma_{\cC^0}(\overline{\cL}).
\end{equation}

For a complex vector space $V$, we denote by $\overline{V}$ the $\bC$-vector space whose underlying set is $V$ with the $\bC$-module structure given by multiplication by the complex conjugate. We get a $\bC$-linear map:
\begin{align}\label{eqn:Clinearconj}
    \overline{\Gammahol\left(\cL \right)} \hookrightarrow \overline{\Gamma_{\cC^0}\left(\cL \right)} \overset{\eqref{eqn:complexconjsections}}{\lra} \Gamma_{\cC^0}(\overline{\cL}).
\end{align}
We let
\begin{align}\label{eqn:defeta}
    \eta := \sum_{j=0}^M z_j \otimes \overline{z_j} \in \Gammahol\left( \cL \right) \otimes_\bC \overline{\Gammahol\left( \cL \right)}.
\end{align}
Its image via the composition of the the map~\eqref{eqn:Clinearconj} and the multiplication map $\Gamma_{\cC^0}(\cL) \otimes_\bC \Gamma_{\cC^0}(\overline{\cL}) \to \Gamma_{\cC^0}(\cL \otimes \overline{\cL})$ is a never vanishing section. It therefore gives an explicit trivialisation of the smooth complex line bundle $\cL \otimes \overline{\cL} \cong X \times \bC$. In particular, we obtain an isomorphism on the level of continuous sections
\begin{align}\label{eqn:trivialisationsections}
    \Gamma_{\cC^0}(\cL \otimes \overline{\cL}) \cong \Gamma_{\cC^0}(X \times \bC) = \cC^0(X, \bC).
\end{align}

\subsection{Stabilisation}

For every integer $k \geq 0$, we now construct the following commutative diagram.
\begin{equation}\label{eqn:stabilisationdiagram}
\begin{tikzcd}
\Gammahol\left( (J^r\cE) \otimes \cL^k \right) \otimes_\bC \overline{\Gammahol\left(\cL^k \right)} \arrow[rrd, "\varphi_k"] \arrow[dd, "\gamma_k"'] &  &                                          \\
&  & \Gamma_{\cC^0}\left( J^r\cE \right) \\
\Gammahol\left( (J^r\cE) \otimes \cL^{k+1} \right) \otimes_\bC \overline{\Gammahol\left(\cL^{k+1} \right)} \arrow[rru, "\varphi_{k+1}"']                &  &                                         
\end{tikzcd}
\end{equation}
The horizontal maps are given by the composition
\begin{equation}\label{eqn:phikmap}
\begin{split}
    \varphi_k \colon& \Gammahol\left( (J^r\cE) \otimes \cL^k \right) \otimes_\bC \overline{\Gammahol\left(\cL^k \right)} \\
        &\lra \Gamma_{\cC^0}\left( J^r\cE \otimes \cL^k \right) \otimes_\bC \Gamma_{\cC^0}\left( \overline{\cL}^k \right) \\
        &\lra \Gamma_{\cC^0} \left(J^r\cE \otimes \cL^k \otimes \overline{\cL}^k \right) \cong \Gamma_{\cC^0}\left(J^r\cE \right)
\end{split}
\end{equation}
where the first arrow is induced by the map~\eqref{eqn:Clinearconj}, the second arrow is the multiplication map, and the last isomorphism is~\eqref{eqn:trivialisationsections} applied to $\left( \cL \otimes \overline{\cL} \right)^k \cong \cL^k \otimes \overline{\cL}^k$.

We construct the vertical map in the diagram~\eqref{eqn:stabilisationdiagram} as the composition:
\begin{equation}\label{eqn:stabilisationmap}
\begin{split}
    \gamma_k \colon &\Gammahol\left( (J^r\cE) \otimes \cL^k \right) \otimes_\bC \overline{\Gammahol\left(\cL^k \right)} \\
    &\quad \lra \Gammahol\left( (J^r\cE) \otimes \cL^k \right) \otimes_\bC \overline{\Gammahol\left(\cL^k \right)} \otimes_\bC \left( \Gammahol\left( \cL \right) \otimes_\bC \overline{\Gammahol\left( \cL \right)} \right) \\
    &\quad \cong \left( \Gammahol\left( (J^r\cE) \otimes \cL^k \right) \otimes_\bC \Gammahol\left( \cL \right) \right) \otimes_\bC \left( \overline{\Gammahol\left(\cL^k \right)} \otimes_\bC \overline{\Gammahol\left( \cL \right)} \right) \\
    &\quad \lra \Gammahol\left( (J^r\cE) \otimes \cL^{k+1} \right) \otimes_\bC \overline{\Gammahol\left(\cL^{k+1} \right)}
\end{split}
\end{equation}
where the first arrow is given by tensoring with the element $\eta$ defined in~\eqref{eqn:defeta}, the isomorphism is given by reordering the factors, and the last arrow is given by the multiplication maps.

The commutativity of the diagram~\eqref{eqn:stabilisationdiagram} follows directly from the fact that $\eta$ is sent to the constant function equal to $1$ via the isomorphism~\eqref{eqn:trivialisationsections}. Loosely speaking, the vertical map $\gamma_k$ is a ``multiplication by $\eta$'', which amounts to multiplying a continuous section of $J^r\cE$ by the constant function $1$ after using the chosen identification~\eqref{eqn:trivialisationsections}.

\begin{example}\label{example:antihomogeneous_stabilisation}
It is illuminating to think about the case $X = \bC\bP^n$, $\cL = \cO(1)$ and $\cE = \cO(d+1)$. In this example, $\Gammahol(\cE)$ is the space of homogeneous polynomials of degree $d+1$ in $n+1$ variables. One may also prove an isomorphism $J^1(\cO(d+1)) \cong \cO(d)^{\oplus (n+1)}$ as holomorphic vector bundles. (See~\cite[Proposition 2.2]{di_rocco_line_2000} for a proof.)

We may then view $\Gammahol\left( (J^1\cE) \otimes \cL^k \right) \otimes_\bC \overline{\Gammahol\left(\cL^k \right)}$ as the space of $(n+1)$-uples of homogeneous polynomials of bi-degree $(d+k,k)$: that is of degree $d+k$ in the variables $z_i$ and of degree $k$ in the complex conjugate variables $\overline{z_i}$. In this case, the image of $\eta$ in $\Gamma_{\cC^0}\left(\cL \otimes \overline{\cL}\right)$ is $|z|^2 := z_0 \overline{z_0} + \cdots + z_n \overline{z_n}$. The isomorphism $\Gamma_{\cC^0}\left(\cL \otimes \overline{\cL}\right) \cong \cC^0(X,\bC)$ corresponding to~\eqref{eqn:trivialisationsections} sends a section $s$ to the map
\[
    z = [z_0:\ldots:z_n] \in \bC\bP^n \longmapsto \frac{s(z)}{|z|^2} \in \bC.
\]
Under these identifications, the map $\gamma_k$ is then:
\[
    \left( f_0, \ldots, f_n \right) \longmapsto \left( (z_0\overline{z_0} + \cdots + z_n\overline{z_n})\cdot f_0, \ldots, (z_0\overline{z_0} + \cdots + z_n\overline{z_n})\cdot f_n \right)
\]
which sends a tuple of polynomials of bi-degree $(d+k,k)$ to one of bi-degree $(d+k+1,k+1)$. (Compare \cite{mostovoy_spaces_2006} for a related situation.)
\end{example}

\bigskip

We will need the following small result, analogous to Lemma~\ref{lemma:multijet}. Let $(x_0, \ldots, x_p)$ be a tuple of points in $X$. We may evaluate a continuous section of $J^r\cE$ simultaneously at all these points:
\begin{equation}\label{eqn:multievalcontinuous}
\begin{split}
    \mathrm{ev}_{(x_0,\ldots,x_p)} \colon \Gamma_{\cC^0}\left( J^r\cE \right) &\lra (J^r\cE)|_{x_0} \times \cdots \times (J^r\cE)|_{x_p} \\
            s &\longmapsto \left( s(x_0), \ldots, s(x_p)\right).
\end{split}
\end{equation}
\begin{lemma}\label{lemma:multieval}
Let $\cE$ be a holomorphic vector bundle on X and $N(\cE,r) \in \bN$ be as in Definition~\ref{def:bigN}. Let $(x_0, \ldots, x_p)$ be a tuple of $p+1$ distinct points in $X$. If $p \leq N(\cE,r)$, the composition
\[
    \Gammahol\left( (J^r\cE) \otimes \cL^k \right) \otimes_\bC \overline{\Gammahol\left(\cL^k \right)} \overset{\varphi_k}{\lra} \Gamma_{\cC^0}\left( J^r\cE \right) \lra (J^r\cE)|_{x_0} \times \cdots \times (J^r\cE)|_{x_p}
\]
of the map $\varphi_k$ of~\eqref{eqn:phikmap} and the simultaneous evaluation~\eqref{eqn:multievalcontinuous} is surjective.
\end{lemma}

\begin{proof}
The case $k=0$ is a direct consequence of Lemma~\ref{lemma:multijet}. The result for $k \geq 1$ then follows from the commutativity of the diagram~\eqref{eqn:stabilisationdiagram}.
\end{proof}

\subsection{Non-singular sections}

We define
\[
    \cN(k) \subset \Gammahol\left( (J^r\cE) \otimes \cL^k \right) \otimes_\bC \overline{\Gammahol\left(\cL^k \right)} 
\]
to be subspace of elements sent to non-singular sections of $J^r\cE$ (as in Definition~\ref{def:singular}) under the map $\varphi_k$ defined in~\eqref{eqn:phikmap}. We say that an $s \in \Gammahol\left( (J^r\cE) \otimes \cL^k \right) \otimes_\bC \overline{\Gammahol\left(\cL^k \right)}$ is non-singular if it is in the subspace $\cN(k)$. We define the singular subset to be the complement
\[
    \cS(k) := \left(\Gammahol\left( (J^r\cE) \otimes \cL^k \right) \otimes_\bC \overline{\Gammahol\left(\cL^k \right)}\right) - \cN(k).
\]

\begin{remark}
When $k = 0$, $\cN(0) \subset \Gammahol(J^r\cE)$ is the usual subspace of non-singular sections of $J^r\cE$ as in Definition~\ref{def:singular}.
\end{remark}

\begin{example}\label{example:antihomogeneous}
In the case $X = \bC\bP^n$, $\cL = \cO(1)$ and $\cE = \cO(d+1)$, recall from Example~\ref{example:antihomogeneous_stabilisation} that the space $\Gammahol\left( (J^1\cE) \otimes \cL^k \right) \otimes_\bC \overline{\Gammahol\left(\cL^k \right)}$ corresponds to $(n+1)$-uples of homogeneous polynomials of degree $d+k$ in the holomorphic variables $z_i$ and of degree $k$ in the complex conjugate variables $\overline{z_i}$. Under this identification, if the Taylor condition $\fT \subset J^1(\cO(d+1))$ is the zero section, the space of non-singular sections $\cN(k)$ contains exactly those $(n+1)$-uples of polynomials that never vanish simultaneously.
\end{example}

\subsection{Resolution and spectral sequence}\label{subsection:adaptationSS}

We now explain how the results from Section~\ref{section:resolution} can be adapted to the case
\[
    \Gamma = \Gammahol\left( (J^r\cE) \otimes \cL^k \right) \otimes_\bC \overline{\Gammahol\left(\cL^k \right)} \quad \text{ and } \quad \Sigma = \cS(k)
\]
to construct a resolution of $\cS(k)$ and a spectral sequence converging to its cohomology, or equivalently to the homology of $\cN(k)$ by Alexander duality. In this case, the definition of the singular space~\eqref{eqn:defsingularset} of $f \in \Gamma$ has to be changed to 
\[
    \mathrm{Sing}(f) := \left\{ x \in X \mid \varphi_k(f)(x) \in \fT \right\} \subset X.
\]
In particular, in the case $k=0$, it agrees with Definition~\ref{def:singular}. The topological results about the resolution just follow from the fact that $\fT \subset J^r\cE$ is closed. In particular, Lemma~\ref{lemma:Xnisclosed} still holds with its proof nearly unchanged: one has to replace the jet map $j^r$ by $\varphi_k$. The construction of the spectral sequence is then unchanged.

\bigskip

The computations of cohomology groups on the $E^1$-page from Section~\ref{section:cohomologygroups} can also be adapted in this case. We first describe what to adapt for the first steps of the filtration. The analogue of Lemma~\ref{lemma:secondbundle} with the jet map $j^r$ replaced by $\varphi_k$ still holds as the key point is the surjectivity established in Lemma~\ref{lemma:multieval}. The other result, Lemma~\ref{lemma:firstbundle}, remains unchanged. Hence, Proposition~\ref{prop:firstpagethomiso} is true in our new setting. 

The adaptations are similar to examine the last step $R^N_\text{cone}\fX - R^N\fX$. Indeed, the same stratification works, as well as the cohomological dimension estimates. In details, Lemma~\ref{lemma:firstestimationFilj} is unchanged, and Lemma~\ref{lemma:YNcohodimestimation} is proved similarly by just replacing the jet map by $\varphi_k$. The other two results, Lemma~\ref{lemma:crucialestimationlemma} and Lemma~\ref{lemma:estimationcohodimapexofcone}, also hold when rewriting the proof by changing the jet map $j^r$ by $\varphi_k$. Indeed, the key ingredients were the semi-algebraicity of the Taylor condition $\fT$ (which remains unchanged), and the fact that the jet map was complex algebraic, hence real semi-algebraic. The map $\varphi_k$ is no longer complex algebraic, but is given by a ratio of algebraic maps and complex conjugates of algebraic maps. In particular, it is real semi-algebraic. This is enough for the proof to go through.

\bigskip

To sum up, we have the following analogue of Proposition~\ref{prop:summingUPresolutionandSS}.
\begin{proposition}\label{prop:summingUPadapted}
Let $\cE$ be a holomorphic vector bundle on $X$ and $\fT \subset J^r\cE$ be an admissible Taylor condition. Let 
\[
    \Gamma = \Gammahol\left( (J^r\cE) \otimes \cL^k \right) \otimes_\bC \overline{\Gammahol\left(\cL^k \right)}
\]
and $\cN(k) \subset \Gamma$ be the subspace of non-singular sections. Let $N = N(\cE,r)$. The resolution and its filtration described in Section~\ref{section:resolution} give rise to a spectral sequence on the second quadrant $s \leq -1$ and $t \geq 0$ converging to the homology of the space of non-singular sections:
\[
    E^1_{s,t} = \CCH^{2\dim_\bC \Gamma - 1 - s - t}(F_{-s-1} - F_{-s-2}; \bZ) \implies \widetilde{H}_{s+t}(\cN(k); \bZ).
\]
The differentials on the $r$-th page have bi-degree $(-r,r-1)$. Furthermore, for $-N-1 \leq s \leq -1$, we have the following isomorphisms for all $t \geq 0$:
\[
    E^1_{s,t} \cong \CCH^{-t - 2s \cdot \rank_\bC J^r\cE}(\fT^{(-s)}; \bZ^\mathrm{sign}).
\]
Moreover, for $t < N\cdot e(\fT) + e(\fT)$:
\[
    E^1_{-N-2,t} = 0.
\]
\end{proposition}

Lastly, let us mention that in the particular example where $X = \bC\bP^n$, $\cL = \cO(1)$, $\cE = \cO(d+1)$ and $\fT \subset J^1\cE$ is the zero section, the spectral sequence is completely analogous to that of~\cite{mostovy_truncated_2012}.

\section{Comparison of spectral sequences}\label{section:compareSS}

From our definition of non-singularity, it follows that the jet map $j^r$ sends a non-singular section $f$ of $\cE$ to a non-singular section $j^r(f)$ of $J^r\cE$. Likewise, the stabilisation map described in~\eqref{eqn:stabilisationmap} sends elements in $\cN(k)$ to elements in $\cN(k+1)$. We shall see that these maps induce isomorphisms in homology in a range of degrees up to around $N = N(\cE,r)$. We first explain the argument for the jet map $j^r$ and then go through the required modifications for the stabilisation map.

\subsection{The case of the jet map}

Reading Proposition~\ref{prop:summingUPresolutionandSS} and Proposition~\ref{prop:summingUPadapted}, we may observe that we have similar looking spectral sequences, one converging to the homology of $\Gamma_{\mathrm{hol,ns}}(\cE)$ and the other one to that of $\Gamma_{\mathrm{hol,ns}}(J^r\cE)$. In particular, in the range $-N-1 \leq s \leq -1$, the terms $E^1_{s,t}$ are given by the same cohomology groups 
\[
    E^1_{s,t} \cong \CCH^{-t - 2s \cdot \rank_\bC J^r\cE}(\fT^{(-s)}; \bZ^\mathrm{sign})
    \]
in both spectral sequences. If we had a morphism of spectral sequences that happened to be an isomorphism in this range, then, using the vanishing result $E^1_{-N-2,t} = 0$ for $t < N\cdot e(\fT) + e(\fT)$, the morphism induced on the $E^\infty$-page would be an isomorphism in the range of degrees $* < N(e(\fT)-1) + e(\fT) - 2$. (See Figure~\ref{fig:firstpagefigure} where we have drawn some differentials.) We shall construct such a morphism of spectral sequences, whilst making sure that it is compatible with the morphism induced on homology by the jet map $j^r$:
\[
    \widetilde{H}_{s+t}(\Gamma_{\mathrm{hol,ns}}(\cE)) \lra \widetilde{H}_{s+t}(\Gamma_{\mathrm{hol,ns}}(J^r\cE)).
\]
For the sake of completeness, we recall when a morphism is compatible with a morphism of spectral sequences. (See e.g. \cite[Section 5.2]{weibel_introduction_1994}.) If two spectral sequences $E^r_{p,q}$ and $E^{'r}_{p,q}$ converge respectively to $H_*$ and $H'_*$, we say that a map $h \colon H_* \to H'_*$ is \emph{compatible} with a morphism $f \colon E \to E'$ if $h$ maps $F_p H_n$ to $F_p H'_n$ (here $F_p$ denotes the filtration) and the associated maps $F_p H_n / F_{p-1} H_n \to F_p H'_n / F_{p-1} H'_n$ correspond to $f^\infty_{p,q} \colon E^\infty_{p,q} \to E^{'\infty}_{p,q}$ (where $q = n-p$) under the isomorphisms $E^\infty_{p,q}\cong F_p H_n / F_{p-1} H_n$ and $E^{'\infty}_{p,q} \cong F_p H'_n / F_{p-1} H'_n$. The main point being that if $f$ is an isomorphism in a range, then $h$ also is an isomorphism in a range. (See \cite[Comparison Theorem 5.2.12]{weibel_introduction_1994}.)

\bigskip

Let $d_1 := 2\dim_\bC \Gammahol(\cE)$ and $d_2 := 2\dim_\bC \Gammahol(J^r\cE)$ be the real dimensions of the complex vector spaces of sections. We define the shriek morphism $j^!$ as the unique morphism making the following square commutative:
\begin{equation}\label{eqn:diagramdefshriekmapAlexanderduality}
\begin{tikzcd}
{\widetilde{H}_*(\Gamma_{\mathrm{hol,ns}}(\cE))} \arrow[d, "\cong"'] \arrow[rr, "(j^r)_*"]      &  & {\widetilde{H}_*(\Gamma_{\mathrm{hol,ns}}(J^r\cE))} \arrow[d, "\cong"]     \\
{\CCH^{d_1 - 1 - *}(\Gammahol(\cE) - \Gamma_{\mathrm{hol,ns}}(\cE))} \arrow[rr, "j^!"', dashed] &  & {\CCH^{d_2 - 1 - *}(\Gammahol(J^r\cE) - \Gamma_{\mathrm{hol,ns}}(J^r\cE))}
\end{tikzcd}
\end{equation}
where the vertical isomorphisms are given by Alexander duality and the top map is induced by the jet map $j^r$ in homology. As our spectral sequences actually converge to the \v{Cech} cohomology with compact support of the singular subspaces, we will construct our morphism of spectral sequences such that it is compatible with $j^!$. 

\bigskip

The spectral sequences arose from filtrations, so we now recall some notation from Section~\ref{section:resolution}. We let $\fX$ be the functor $\catF^\op \to \catTop$ constructed there using $\Gamma = \Gammahol(\cE)$ and $\Sigma = \Gammahol(\cE) - \Gamma_{\mathrm{hol,ns}}(\cE)$. As we have explained in Section~\ref{subsection:adaptationSS}, the resolution also works for $\Gammahol(J^r\cE)$ and its singular subspace, and we let $\fY \colon \catF^\op \to \catTop$ be the associated functor in this case. We denote the filtration of $R^N_\text{cone}\fX$ by
\[
    F_{-1}^1 = \emptyset \subset F_0^1 = R^0\fX \subset \cdots \subset F_N^1 = R^N\fX \subset F_{N+1}^1 = R^N_\text{cone}\fX,
\]
and the analogous one of $R^N_\text{cone}\fY$ by
\begin{equation}\label{eqn:filtrationFrakY}
    F_{-1}^2 = \emptyset \subset F_0^2 = R^0\fY \subset \cdots \subset F_N^2 = R^N\fY \subset F_{N+1}^2 = R^N_\text{cone}\fY.
\end{equation}
We will slightly abuse notation and also write
\begin{equation}\label{eqn:shriekmapRNcone}
    j^! \colon \CCH^*(R^N_\text{cone}\fX) \to \CCH^{*+d_2-d_1}(R^N_\text{cone}\fY)
\end{equation}
for the bottom map defined by making the following square commutative:
\begin{center}
\begin{tikzcd}
{\CCH^*(\Gammahol(\cE) - \Gamma_{\mathrm{hol,ns}}(\cE))} \arrow[d, "\cong"'] \arrow[rr, "j^!"] &  & {\CCH^{*+d_2-d_1}(\Gammahol(\cE) - \Gamma_{\mathrm{hol,ns}}(J^r\cE))} \arrow[d, "\cong"] \\
\CCH^*(R^N_\text{cone}\fX) \arrow[rr, "j^!"', dashed]                                          &  & \CCH^{*+d_2-d_1}(R^N_\text{cone}\fY)                                                    
\end{tikzcd}
\end{center}
Recall from the general theory that the spectral associated to the filtration $F_*^i$, $i=1,2$, arises from an exact couple $(\CCH^\bullet(F_*^i), \CCH^\bullet(F_*^i - F_{*-1}^i))$. The map of spectral sequences that we want is then constructed via a map of exact couples as in the following lemma.
\begin{lemma}\label{lemma:constructingmorphismexactcouple}
Let $\delta = d_2 - d_1 = 2(\dim_\bC \Gammahol(\cE) - \dim_\bC \Gammahol(J^r\cE))$. There exists a morphism of exact couples
\[
    \left(j^!_p, \ j^!_{(p)} \right)_{p \geq 0} \colon \left(\CCH^*(F_p^1),  \CCH^*(F_p^1 - F_{p-1}^1)\right) \lra \left(\CCH^{*+\delta}(F_p^2),  \CCH^{*+\delta}(F_p^2 - F_{p-1}^2)\right)
\]
satisfying the following two assertions:
\begin{enumerate}
    \item For $0 \leq p \leq N$, the map $j^!_{(p)}$ in the following diagram is an isomorphism:
    \begin{equation}\label{eqn:thomisolemmamorphismSScommutes}
    \begin{tikzcd}
    \CCH^*(F_p^1 - F_{p-1}^1) \arrow[d, "j^!_{(p)}"'] \arrow[rr, "\cong"] &  & \CCH^{\bullet}(\fT^{(p+1)}; \bZ^\mathrm{sign}) \\
    \CCH^{*+\delta}(F_p^2 - F_{p-1}^2) \arrow[rr, "\cong"']               &  & \CCH^{\bullet}(\fT^{(p+1)}; \bZ^\mathrm{sign})
    \end{tikzcd}
    \end{equation}
    where
    \[
        \bullet = * -2\dim_\bC \Gammahol(\cE) - p + 2(p+1)\rank_\bC J^r\cE
    \]
    and the horizontal isomorphisms are given by Thom isomorphisms as in Proposition~\ref{prop:firstpagethomiso}.

    \item The map $j^!_{N+1}$ is equal to the shriek map~\eqref{eqn:shriekmapRNcone}.
\end{enumerate}
\end{lemma}

Unpacking the definition of a morphism of exact couples, we see that it amounts to providing morphisms $j^!_p$ and $j^!_{(p)}$ for $0 \leq p \leq N+1$ such that the following diagram commutes
\begin{equation*}
\begin{tikzcd}
\CCH^{*-1}(F_{p-1}^1) \arrow[r] \arrow[d, "j^!_{p-1}"] & \CCH^*(F_p^1 - F_{p-1}^1) \arrow[r] \arrow[d, "j^!_{(p)}"] & \CCH^*(F_p^1) \arrow[r] \arrow[d, "j^!_p"] & \CCH^*(F_{p-1}^1) \arrow[d, "j^!_{p-1}"] \\
\CCH^{*-1+\delta}(F_{p-1}^2) \arrow[r]                 & \CCH^{*+\delta}(F_p^2 - F_{p-1}^2) \arrow[r]               & \CCH^{*+\delta}(F_p^2) \arrow[r]           & \CCH^{*+\delta}(F_{p-1}^2)              
\end{tikzcd}
\end{equation*}
where the horizontal morphisms in the diagram are given by the long exact sequence of the pair $(F_p^i,F_{p-1}^i)$ for $i=1,2$. 

\bigskip

This result says exactly what we need: there a morphism of spectral sequences compatible with $j^!$ (by (2)) and giving an isomorphism in the vertical strip $-N-1 \leq s \leq 1$ (by (1)). The lemma, as well as the strategy of proof, is adapted from~\cite[Proposition 4.7]{vokrinek_generalization_2007}. First, let us state the most important consequence:
\begin{proposition}\label{prop:jetmapisohomology}
For a holomorphic vector bundle $\cE$ on $X$, the jet map 
\[
    j^r \colon \Gamma_{\mathrm{hol,ns}}(\cE) \lra \Gamma_{\mathrm{hol,ns}}(J^r\cE)
\]
induces an isomorphism in homology in the range of degrees $* < N(\cE,r)\cdot (e(\fT)-1) + e(\fT) - 2$. \qed
\end{proposition}

\bigskip

To understand how to construct the degree-shifting morphisms of Lemma~\ref{lemma:constructingmorphismexactcouple}, it is helpful to give a description of the shriek map between cohomology groups arising from Alexander duality as in the diagram~\eqref{eqn:diagramdefshriekmapAlexanderduality}. We shall do so generally first (following \cite[Appendix D]{vokrinek_generalization_2007}) and then specialise to our situation to prove the lemma at hand.

\subsubsection{Alexander duality and shriek maps}

Let $p \colon E \to B$ be a vector bundle between oriented paracompact topological manifolds of dimension $n$ and $m$ respectively. Let $j \colon K \subset E$ be a closed subset, and let $i \colon B \hookrightarrow E$ be the zero section. We will see $B$ as a submanifold of $E$ via $i$. Using Alexander duality (the vertical isomorphisms in the diagram below), we may define the \emph{shriek map}
\begin{equation}\label{eqn:generalshriekmap}
    i^! \colon \CCH^*(B \cap K) \to \CCH^{*+(n-m)}(K)
\end{equation}
to be the unique morphism making the following diagram commute:
\begin{center}
\begin{tikzcd}
{H_*(B,B-B \cap K)} \arrow[rr, "i_*"]                             &  & {H_*(E,E-K)}                      \\
\CCH^{m-*}(B\cap K) \arrow[u, "\cong"] \arrow[rr, "i^!"', dashed] &  & \CCH^{n-*}(K) \arrow[u, "\cong"']
\end{tikzcd}
\end{center}
The goal of this section is to give a more intrinsic definition of $i^!$ that will allow us to define the required morphisms in Lemma~\ref{lemma:constructingmorphismexactcouple}.

\bigskip

Firstly, Vok\v{r}\'{i}nek proves in~\cite[Proposition D.1]{vokrinek_generalization_2007} the following:
\begin{lemma}\label{lemma:vokrinekcommutes}
The diagram below commutes:
\begin{center}
\begin{tikzcd}
{H_*(B,B-B \cap K)} \arrow[rr, "i_*"]  &                                                                            & {H_*(E,E-K)}                      \\
\CCH^{m-*}(B\cap K) \arrow[u, "\cong"] & \check{H}_{p^{-1}c}^{m-*}(K) \arrow[l, "k^*"] \arrow[r, "- \cup j^*\tau"'] & \CCH^{n-*}(K) \arrow[u, "\cong"']
\end{tikzcd}
\end{center}
where the vertical isomorphisms are given by Alexander duality, $k \colon B \cap K \hookrightarrow K$ is the inclusion, $\tau \in H^\delta(D(E), S(E))$ is the Thom class of $p$, and $p^{-1}c$ is the family of supports defined as:
\[
    p^{-1}c = \left\{ F \subset K \ \middle| \ F \text{ closed and } \overline{p(F)} \subset B \cap K \text{ is compact} \right\}
\]
so that $\check{H}^*_{p^{-1}c}$ denotes \v{C}ech cohomology with supports in $p^{-1}c$. (See e.g.~\cite[Chapter II.2]{bredon_sheaf_1997}.)
\end{lemma}

\begin{proof}[Sketch of proof]
We repeat Vok\v{r}\'{i}nek's proof here for convenience. First, we explain the morphisms in Alexander duality. Recall from e.g.~\cite[Corollary V.10.2]{bredon_sheaf_1997} that we have fundamental classes $[B] \in \BMH_m(B)$ and $[E] \in \BMH_n(E)$, where $\BMH_*$ denotes Borel--Moore homology (also known as homology with closed support). Using the proper inclusions $(E,\emptyset) \hookrightarrow (E, E-K)$ and $(B,\emptyset) \hookrightarrow (B, B - B\cap K)$, they give rise to classes $o_E \in \BMH_n(E, E-K)$ and $o_B \in \BMH_m(B, B- B\cap K)$. If $U \subset E$ is a closed neighbourhood of $K$, we get a morphism
\[
    \CCH^{n-*}(U) \overset{- \cap o_E|U}{\lra} H_*(U, U-K) \lra H_*(E, E-K)
\]
where $o_E|U$ is the image of $o_E$ via the excision isomorphism $\BMH_n(E, E-K) \cong \BMH_n(U, U-K)$. (Note that it is important for $U$ to be closed, so that the inclusion $U \hookrightarrow E$ is proper, hence induces a morphism in Borel--Moore homology.) Likewise, we get a morphism
\[
    \CCH^{m-*}(B \cap U) \overset{- \cap o_B|U}{\lra} H_*(B \cap U, B \cap (U-K)) \lra H_*(B, B- B\cap K).
\]
Now, the isomorphisms in Alexander duality are given by taking the colimit over all closed neighbourhoods $U$ of $K$ of the two morphisms constructed above. (This is explained in~\cite[V.9]{bredon_sheaf_1997}.) Hence, to prove the lemma, it suffices to check commutativity of the following diagram:
\begin{center}
\begin{tikzcd}
{H_*(B \cap U,B \cap (U-K))} \arrow[rr, "g_*"] &                                                                                                                        & {H_*(U,U-K)}                           \\
\CCH^{m-*}(B\cap U) \arrow[u, "- \cap o_B|U"]  & \check{H}_{p^{-1}c}^{m-*}(U) \arrow[l, "g^*"] \arrow[r, "- \cup h^*\tau"'] \arrow[ru, "- \cap g_*(o_B|U)" description] & \CCH^{n-*}(U) \arrow[u, "- \cap o_E"']
\end{tikzcd}
\end{center}
where $g \colon B \cap U \hookrightarrow U$ and $h \colon U \hookrightarrow E$ are the inclusions. The left part commutes by naturality of the cap products. The right part commutes by observing that the fundamental classes can be chosen to correspond under the Thom isomorphism, which implies that $h^*\tau \cap o_E|U = g_* o_B|U$, and finishes the proof.
\end{proof}
In the statement of Lemma~\ref{lemma:vokrinekcommutes}, if the morphism $k^*$ were invertible, the shriek map~\eqref{eqn:generalshriekmap} would be given by ``$(k^*)^{-1}$'' followed by taking the cup product with the ``Thom class'' $j^*\tau$. However, it is not invertible in general. There is nevertheless a way around that problem which we explain below, using $\varepsilon$-small neighbourhoods of $B \cap K$ in $K$ and the continuity property of \v{C}ech cohomology.

\bigskip

We choose, once and for all, a bundle metric on $p \colon E \to B$. For a real number $\varepsilon > 0$, denote by $D_\varepsilon$ (resp. $S_\varepsilon$, $\mathring{D}_\varepsilon$) the closed disc (resp. sphere, open disc) sub-bundle of $E \to B$ of radius $\epsilon$ (for the chosen metric). In \cite[Lemma D.2]{vokrinek_generalization_2007}, Vok\v{r}\'{i}nek proves:
\begin{lemma}\label{lemma:vokrinekcommutes2}
The following diagram commutes:
\begin{equation}\label{eqn:diagramgeneralsetupVokrinek}
\begin{tikzcd}
{H_*(B,B-B \cap K)} \arrow[rr]         &                                                                                                        & {H_*(E \cap \mathring{D}_\varepsilon, (E-K) \cap \mathring{D}_\varepsilon)} \arrow[r]        & {H_*(E,E-K)}                      \\
\CCH^{m-*}(B\cap K) \arrow[u, "\cong"] &                                                                                                        & \CCH^{n-*}(K \cap \mathring{D}_\varepsilon) \arrow[d, "\cong"] \arrow[r] \arrow[u, "\cong"'] & \CCH^{n-*}(K) \arrow[u, "\cong"'] \\
                                       & \CCH^{m-*}(K \cap D_\varepsilon) \arrow[lu, "(l_\varepsilon)_*"] \arrow[r, "- \cup \tau_\varepsilon"'] & {\CCH^{n-*}(K \cap D_\varepsilon, K \cap S_\varepsilon)}                                     &                                  
\end{tikzcd}
\end{equation}
where the vertical isomorphisms on the first row are given by Alexander duality, the one on the second row follows from general results about cohomology with compact supports, $l_\varepsilon \colon B \cap K \hookrightarrow K \cap D_\varepsilon$ is the inclusion, $\tau_\varepsilon$ is the restriction of the Thom class of $E \to B$, and the rightmost horizontal arrows are induced by the inclusions. (Recall that cohomology with compact supports in covariant for open inclusions.)
\end{lemma}
\begin{proof}[Sketch of proof]
The left part of the diagram can be shown to commute by a proof analogous to that of Lemma~\ref{lemma:vokrinekcommutes}. The right-hand square is seen to commute by a direct verification.
\end{proof}

Taking the limit $\varepsilon \to 0$, the morphisms $(l_\varepsilon)_*$ induce a morphism from the colimit
\[
    \colim\limits_{\epsilon \to 0} \CCH^{m-*}(K \cap D_\varepsilon) \lra \CCH^{m-*}(B\cap K)
\]
which is an isomorphism by the continuity property of \v{C}ech cohomology with compact supports (see, e.g. \cite[Theorem 14.4]{bredon_sheaf_1997} where it is stated using sheaf cohomology which agrees with \v{C}ech cohomology here). We finally obtain another description of the shriek map $i^!$:
\begin{proposition}[Compare {\cite[Theorem D.3]{vokrinek_generalization_2007}}]\label{prop:vokrinekfinal}
The shriek map $i^!$ defined in~\eqref{eqn:generalshriekmap} is equal to the composite obtained as one goes along the bottom path in the diagram~\eqref{eqn:diagramgeneralsetupVokrinek} above, i.e.:
\begin{equation*}
\begin{split}
    i^! \colon \CCH^{m-*}(B\cap K) &\overset{\cong}{\longleftarrow} \colim\limits_{\epsilon \to 0} \CCH^{m-*}(K \cap D_\varepsilon) \\
    &\lra \colim\limits_{\epsilon \to 0} \CCH^{n-*}(K \cap D_\varepsilon, K \cap S_\varepsilon) \cong \colim\limits_{\epsilon \to 0} \CCH^{n-*}(K \cap \mathring{D}_\varepsilon) \\
    &\lra \CCH^{n-*}(K).
\end{split}
\end{equation*}
Furthermore, in the case where both $E$ and $B$ are themselves vector bundles over a same base, $K = E$, and $i \colon B \hookrightarrow E$ is the inclusion of a sub-bundle, the shriek map $i^!$ is the Thom isomorphism of the bundle $E \to B$ given by choosing a splitting of $i$.
\end{proposition}

\begin{proof}
The first part follows from Lemma~\ref{lemma:vokrinekcommutes} and Lemma~\ref{lemma:vokrinekcommutes2}. The second part is shown by direct inspection of the construction.
\end{proof}

\subsubsection{The proof of Lemma~\ref{lemma:constructingmorphismexactcouple}}

We shall apply the general theory described in the last section to our case. To lighten the notation, we write
\[
    \Gamma_1 := \Gammahol(\cE), \quad \Sigma_1 := \Gammahol(\cE) - \Gamma_{\mathrm{hol,ns}}(\cE)
\]
and
\[
    \Gamma_2 := \Gammahol(J^r\cE), \quad \Sigma_2 := \Gammahol(J^r\cE) - \Gamma_{\mathrm{hol,ns}}(J^r\cE).
\]
The jet map $j^r$ gives a linear embedding of $\Gamma_1$ into $\Gamma_2$ such that the image of the singular subspace is precisely given by the intersection with the bigger singular subspace, i.e.
\[
    j^r(\Sigma_1) = j^r(\Gamma_1) \cap \Sigma_2.
\]
Choosing a complementary linear subspace of $j^r(\Gamma_1)$ inside $\Gamma_2$, we obtain a projection giving a vector bundle
\begin{equation}\label{eqn:jrembeddingvectorbundle}
    \Gamma_2 \lra j^r(\Gamma_1) \cong \Gamma_1
\end{equation}
of real rank $\delta = d_2 - d_1$. Below, we apply Vok\v{r}\'{i}nek's results to this situation.

\bigskip 

We first set up the notation. Let $\epsilon > 0$ be a positive real number and denote by $D_\epsilon$ (resp. $S_\epsilon$, $\mathring{D}_\varepsilon$) the closed disc (resp. sphere, open disc) sub-bundle of radius $\epsilon$ of the vector bundle~\eqref{eqn:jrembeddingvectorbundle}. Recall from~\eqref{eqn:filtrationFrakY} the functor $\fY$ giving rise to the resolution of $\Sigma_2$. We also define $\fY_{D_\epsilon} \colon \catF^\op \to \catTop$ to be the sub-functor of $\fY$ given by
\[
    \fY_{D_\epsilon}[n] := \left\{ (f,s_0,\ldots,s_n) \in \fY[n] \mid f \in D_\epsilon \right\}
\]
and likewise for $\fY_{S_\epsilon} \subset \fY$ and $\fY_{\mathring{D}_\varepsilon} \subset \fY$ using only sections $f \in S_\epsilon$ or $\mathring{D}_\varepsilon$. Let $\tau_\varepsilon \in H^\delta(\Sigma_2 \cap D_\varepsilon, \Sigma_2 \cap S_\varepsilon)$ be the restriction of the Thom class of the vector bundle~\eqref{eqn:jrembeddingvectorbundle} to $\Sigma_2$. (Recall that the Thom class is an element of $H^\delta(D_\varepsilon, S_\varepsilon)$.) In all what follows, we see $\Gamma_1 \subset \Gamma_2$ via the embedding $j = j^r$. Let $l_\varepsilon \colon \Sigma_1 \hookrightarrow \Sigma_2 \cap D_\varepsilon$ be the inclusion (which is proper, hence induces a morphism on compactly supported cohomology). We explained in Proposition~\ref{prop:vokrinekfinal} that the shriek map $j^!$ is obtained from the zigzag
\[
    \CCH^*(\Sigma_1) \overset{(l_\varepsilon)_*}{\leftarrow} \CCH^*(\Sigma_2 \cap D_\varepsilon) \overset{-\cup \tau_\varepsilon}{\to} \CCH^{*+\delta}(\Sigma_2 \cap D_\varepsilon, \Sigma_2 \cap S_\varepsilon) \cong \CCH^{*+\delta}(\Sigma_2 \cap \mathring{D}_\varepsilon) \to \CCH^{*+\delta}(\Sigma_2)
\]
by taking a colimit as $\varepsilon \to 0$. 

\bigskip

We mimic that construction at the level of the resolutions. Let $0 \leq p \leq N+1$ be an integer. Recall from~\eqref{eqn:filtrationFrakY} that $F_p^i$ denoted the $p$-th step of the filtration of the resolution of $\Sigma_i$. We denote by $F^2_{p,D_\varepsilon}$, $F^2_{p,S_\varepsilon}$, and $F^2_{p,\mathring{D}_\varepsilon}$ the analogous filtrations on the resolutions obtained from the subfunctors $\fY_{D_\varepsilon}$, $\fY_{S_\varepsilon}$ and $\fY_{\mathring{D}_\varepsilon}$ respectively. 
Because a singular point of a section $f \in \Gamma_1$ is also a singular point of $j^r(f) \in \Gamma_2$, the jet map gives a map on resolutions
\begin{align*}
    \fX[p] &\lra \fY[p] \\
    (f,s_0,\ldots,s_p) &\longmapsto (j^r(f), s_0, \ldots, s_p).
\end{align*}
which preserves the filtrations. Let $\tilde{l_\varepsilon} \colon F_p^1 \hookrightarrow F^2_{p,D_\varepsilon}$ be the induced inclusion. Let $\gamma_\varepsilon \in H^\delta(F^2_{p,D_\varepsilon}, F^2_{p,S_\varepsilon})$ be the pullback of $\tau_\varepsilon$ along $(F^2_{p,D_\varepsilon}, F^2_{p,S_\varepsilon}) \to (\Sigma_2 \cap D_\varepsilon, \Sigma_2 \cap S_\varepsilon)$. The following diagram then commutes by naturality of all the constructions involved:
\begin{center}
\begin{forcedcentertikzcd}
\CCH^*(F^1_p)              & {\CCH^*(F^2_{p,D_\varepsilon})} \arrow[l, "(\tilde{l_\varepsilon})_*"'] \arrow[r, "- \cup \gamma_\epsilon"]        & {\CCH^{*+\delta}(F^2_{p,D_\varepsilon}, F^2_{p,S_\varepsilon}) \cong \CCH^{*+\delta}(F^2_{p,\mathring{D}_\varepsilon})} \arrow[r]                             & \CCH^{*+\delta}(F^2_p)              \\
\CCH^*(\Sigma_1) \arrow[u] & \CCH^*(\Sigma_2 \cap D_\varepsilon) \arrow[u] \arrow[l, "(l_\varepsilon)_*"] \arrow[r, "- \cup \tau_\varepsilon"'] & {\CCH^{*+\delta}(\Sigma_2 \cap D_\varepsilon, \Sigma_2 \cap S_\varepsilon) \cong \CCH^{*+\delta}(\Sigma_2 \cap \mathring{D}_\varepsilon)} \arrow[u] \arrow[r] & \CCH^{*+\delta}(\Sigma_2) \arrow[u]
\end{forcedcentertikzcd}
\end{center}
where all the vertical maps are induced by the proper projections $F^i_p \to \Sigma_i$.
The morphism $j^!_p \colon \CCH^*(F^1_p) \to \CCH^{*+\delta}(F^2_p)$ is then defined as the colimit, when $\varepsilon \to 0$, of the top composition in the diagram above. (Recall that $(\tilde{l_\varepsilon})_*$ is an isomorphism in the colimit, by continuity of \v{C}ech cohomology.) In particular, when $p = N+1$, the vertical map are isomorphisms (by \ref{prop:piinducesiso}), which proves the assertion (2) of Lemma~\ref{lemma:constructingmorphismexactcouple} by noticing that the bottom composition is the shriek map $j^!$.

\bigskip

The morphisms $j^!_{(p)} \colon \CCH^*(F^1_p - F^1_{p-1}) \to \CCH^{*+\delta}(F^2_p - F^2_{p-1})$ are defined analogously, i.e. by the colimit as $\varepsilon \to 0$ of the zig-zag:
\begin{equation*}
\begin{split}
    \CCH^*(F^1_p - F^1_{p-1}) &\longleftarrow \CCH^*(F^2_{p,D_\varepsilon} - F^2_{p-1,D_\varepsilon}) \\
        &\lra \CCH^{*+\delta}(F^2_{p,D_\varepsilon} - F^2_{p-1,D_\varepsilon}, F^2_{p,S_\varepsilon} - F^2_{p-1,S_\varepsilon}) \cong \CCH^{*+\delta}(F^2_{p,\mathring{D}_\varepsilon} - F^2_{p-1,\mathring{D}_\varepsilon}) \\
        &\lra \CCH^{*+\delta}(F^2_p - F^2_{p-1})
\end{split}
\end{equation*}
where, as before, the first morphism is induced by the inclusion, the second morphism is the cup product with the Thom class, and the third is induced covariantly by the open inclusion.

\bigskip

One may check, using naturality of the various constructions involved, that the morphisms $j^!_p$ and $j^!_{(p)}$ give a morphism of exact couples. This amounts to staring at the following commutative diagram.
\begin{center}
\adjustbox{scale=.8,center}{%
\begin{forcedcentertikzcd}
\CCH^{*-1}(F^1_{p-1}) \arrow[r]                                                                     & \CCH^*(F^1_p - F^1_{p-1}) \arrow[r]                                                                                                              & \CCH^*(F^1_p) \arrow[r]                                                                      & \CCH^*(F^1_{p-1})                                                                       \\
{\CCH^{*-1}(F^2_{p-1,D_\varepsilon})} \arrow[d] \arrow[u] \arrow[r]                                 & {\CCH^*(F^2_{p,D_\varepsilon} - F^2_{p-1,D_\varepsilon})} \arrow[u] \arrow[d] \arrow[r]                                                          & {\CCH^*(F^2_{p,D_\varepsilon})} \arrow[u] \arrow[d] \arrow[r]                                & {\CCH^*(F^2_{p-1,D_\varepsilon})} \arrow[u] \arrow[d]                                   \\
{\CCH^{*-1+\delta}(F^2_{p-1,D_\varepsilon}, F^2_{p-1, S_\varepsilon})} \arrow[r] \arrow[d, "\cong"] & {\CCH^{*+\delta}(F^2_{p,D_\varepsilon} - F^2_{p-1,D_\varepsilon}, F^2_{p,S_\varepsilon} - F^2_{p-1,S_\varepsilon})} \arrow[r] \arrow[d, "\cong"] & {\CCH^{*+\delta}(F^2_{p,D_\varepsilon}, F^2_{p,S_\varepsilon})} \arrow[r] \arrow[d, "\cong"] & {\CCH^{*+\delta}(F^2_{p-1, D_\varepsilon}, F^2_{p-1,S_\varepsilon})} \arrow[d, "\cong"] \\
{\CCH^{*-1+\delta}(F^2_{p-1,\mathring{D}_\varepsilon})} \arrow[d] \arrow[r]                         & {\CCH^{*+\delta}(F^2_{p,\mathring{D}_\varepsilon} - F^2_{p-1,\mathring{D}_\varepsilon})} \arrow[d] \arrow[r]                                     & {\CCH^{*+\delta}(F^2_{p,\mathring{D}_\varepsilon})} \arrow[d] \arrow[r]                      & {\CCH^{*+\delta}(F^2_{p-1, \mathring{D}_\varepsilon})} \arrow[d]                        \\
\CCH^{*-1+\delta}(F^2_{p-1}) \arrow[r]                                                              & \CCH^{*+\delta}(F^2_p - F^2_{p-1}) \arrow[r]                                                                                                     & \CCH^{*+\delta}(F^2_p) \arrow[r]                                                             & \CCH^{*+\delta}(F^2_{p-1})                                                             
\end{forcedcentertikzcd}
}
\end{center}

\bigskip

To conclude the proof, we verify the assertion (1) of Lemma~\ref{lemma:constructingmorphismexactcouple}, i.e. that the morphism
\[
    j^!_{(p)} \colon \CCH^*(F^1_p - F^1_{p-1}) \lra \CCH^{*+\delta}(F^2_p - F^2_{p-1})
\]
is an isomorphism. Recall from~\eqref{eqn:filtrationleafhomeo} that 
\[
    F^2_p - F^2_{p-1} \cong Y_p(\fY) \times_{\fS_{p+1}} \mathring{|\Delta^p|} \quad \text{ and } \quad F^1_p - F^1_{p-1} \cong Y_p(\fX) \times_{\fS_{p+1}} \mathring{|\Delta^p|}
\]
where we defined as in~\eqref{eqn:Yjsubspacepairwisedistinct} the subspace
\[
    Y_p(\fY) := \left\{ (f,s_0, \ldots, s_p) \in \fY[p] \mid s_l \neq s_k \text{ if } l \neq k \right\} \subset \fY[p]
\]
and likewise for $Y_p(\fX) \subset \fX[p]$. Recall also that these spaces were vector bundles over $\fT^{(p+1)}$. (See Section~\ref{section:cohomologygroups}.) Hence, we have an inclusion of vector bundles:
\begin{center}
\begin{tikzcd}
F^1_p - F^1_{p-1} \arrow[rd] \arrow[rr, hook] &             & F^2_p - F^2_{p-1} \arrow[ld] \\
                                              & \fT^{(p+1)} &                             
\end{tikzcd}
\end{center}
Now, the second part of Proposition~\ref{prop:vokrinekfinal} applies and finishes the proof.
\qed

\subsection{The case of the stabilisation map}

Choose some integer $k \geq 0$. We now describe how the argument of the previous section can be made with the stabilisation map 
\[
    \gamma_k \colon \Gammahol\left( (J^r\cE) \otimes \cL^k \right) \otimes_\bC \overline{\Gammahol\left(\cL^k \right)} \lra \Gammahol\left( (J^r\cE) \otimes \cL^{k+1} \right) \otimes_\bC \overline{\Gammahol\left(\cL^{k+1} \right)}
\]
from~\eqref{eqn:stabilisationmap}. First of all, it is a linear embedding. Hence, by choosing a complementary subspace, we get a vector bundle
\[
    \Gammahol\left( (J^r\cE) \otimes \cL^{k+1} \right) \otimes_\bC \overline{\Gammahol\left(\cL^{k+1} \right)} \lra \gamma_k\left( \Gammahol\left( (J^r\cE) \otimes \cL^k \right) \otimes_\bC \overline{\Gammahol\left(\cL^k \right)} \right)
\]
analogous to the one in~\eqref{eqn:jrembeddingvectorbundle}. From the commutativity of the diagram~\eqref{eqn:stabilisationdiagram}, we see that a singularity $x \in X$ for $f \in \cS(k)$ is also a singularity of $\gamma_k(f) \in \cS(k+1)$. Therefore, we also get a map induced on the respective resolutions of $\cS(k)$ and $\cS(k+1)$. Together with the fact that non-singular sections are sent to non-singular sections, this is enough for the argument to be repeated in that case.

\begin{proposition}\label{prop:stabilisationisohomology}
The restriction of the stabilisation map $\gamma_k$ to the non-singular subspaces
\[
    \gamma_k \colon \cN(k) \lra \cN(k+1)
\]
induces an isomorphism in homology in the range of degrees $* < N(\cE,r)\cdot (e(\fT)-1) + e(\fT) - 2$. \qed
\end{proposition}

Combining Proposition~\ref{prop:jetmapisohomology} and Proposition~\ref{prop:stabilisationisohomology}, we obtain the following.
\begin{proposition}\label{prop:combinedisohomology}
Each map in the composition
\[
    \Gamma_{\mathrm{hol,ns}}(\cE) \lra \Gamma_{\mathrm{hol,ns}}(J^r\cE) = \cN(0) \lra \colim\limits_{k \to \infty} \cN(k)
\]
induces an isomorphism in homology in the range of degrees $* < N(\cE,r)\cdot (e(\fT)-1) + e(\fT) - 2$. \qed
\end{proposition}

\section{Comparison of holomorphic and continuous sections}\label{section:holcontinuouscompare}

We shall relate $\colim_k \cN(k)$ to the space  $\Gamma_{\cC^0, \mathrm{ns}}\left(J^r\cE\right)$ of non-singular continuous sections of the jet bundle. Recall from the stabilisation diagram~\eqref{eqn:stabilisationdiagram} that every non-singular space $\cN(k)$ maps via $\varphi_k$ to $\Gamma_{\cC^0, \mathrm{ns}}\left(J^r\cE\right)$. The aim of this section is to prove the following result about the map induced from the colimit.

\begin{proposition}\label{prop:weakequivalence}
The map 
\begin{equation}\label{eqn:mapfromcolimNk}
    \colim\limits_{k \to \infty} \cN(k) \lra \Gamma_{\cC^0, \mathrm{ns}}\left(J^r\cE\right)
\end{equation}
is a weak homotopy equivalence.
\end{proposition}

Combining this result with Proposition~\ref{prop:combinedisohomology} readily implies Theorem~\ref{thm:mainthm}. Proposition~\ref{prop:weakequivalence} is a direct consequence of the openness of the subspace of non-singular sections, which follows from the fact that the admissible Taylor condition $\fT \subset J^r\cE$ is closed (see the discussion after Lemma~\ref{lemma:rhonisproper}), and the following
\begin{lemma}\label{lemma:densitylemma}
Let $F$ be a finite CW-complex. The map
\[
    \cC^0\big(F, \ \colim\limits_{k \to \infty} \cN(k)\big) \lra \cC^0\big(F, \ \Gamma_{\cC^0, \mathrm{ns}}\left(J^r\cE\right)  \big)
\]
induced by~\eqref{eqn:mapfromcolimNk} has a dense image.
\end{lemma}

As in~\cite{mostovoy_spaces_2006}, we will need an adaptation of the classical Stone--Weierstrass theorem for real vector bundles.
\begin{theorem}[Stone--Weierstrass]\label{thm:StoneWeiestrass}
Let $E \to B$ be a finite rank real vector bundle over a compact Hausdorff space. Let $A \subset \cC^0(B,\bR)$ be a subalgebra and $\{s_j\}_{j\in J}$ be a set of sections such that
\begin{enumerate}
    \item the subalgebra $A$ separates the points of $B$: for any $x,y \in B$, there exists $h \in A$ such that $h(x) \neq h(y)$;
    \item for any $x \in B$, there exists $h \in A$ such that $h(x) \neq 0$;
    \item for any $x\in B$, the fibre $E_x$ is spanned by the $s_j(x)$ as an $\bR$-vector space.
\end{enumerate}
Then the $A$-module generated by the $s_j$ is dense for the sup-norm (induced by the choice of any inner product on $E$) in the space of all continuous sections of $E$.
\end{theorem}

\begin{proof}[Proof of Lemma~\ref{lemma:densitylemma}]
Let $F$ be a finite CW-complex. By adjunction, a continuous map $F \to \Gamma_{\cC^0, \mathrm{ns}}\left(J^r\cE\right)$ corresponds to a section of the underlying real vector bundle of $J^r\cE \times F \to X \times F$. We shall apply Theorem~\ref{thm:StoneWeiestrass} to that vector bundle. 

Recall that we have chosen in Section~\ref{section:interpolation} a very ample line bundle $\cL$ on $X$ and explained how to define the complex conjugate $\overline{s}$ of a section $s$ of $\cL$. For any integer $k \geq 0$, define the squared norm of a holomorphic section of $\cL$ by
\begin{align*}
    |\cdot|^2 \colon \Gammahol(\cL^k) &\lra \Gamma_{\cC^0}(\cL^k \otimes \overline{\cL}^k) \cong \cC^0(X,\bC) \\
    s &\longmapsto |s|^2 := s\overline{s}
\end{align*}
where the isomorphism with continuous maps was obtained in~\eqref{eqn:trivialisationsections}. Notice that $|s|^2$ is in fact a real valued function $X \to \bR \subset \bC$. We also let
\[
    A_k := \left\{ |g(\cdot,\cdot)|^2 \colon X \times F \to \bR \mid g \in \cC^0(F, \Gammahol(\cL^k)) \right\} \subset \cC^0(X\times F, \bR)
\]
where if $g \in \cC^0(F, \Gammahol(\cL^k))$, we see $g(\cdot,\cdot)$ as a map from  $X \times F$ to $\cL^k$ by adjunction. Keeping the notation from Theorem~\ref{thm:StoneWeiestrass}, we let $A$ to be the subalgebra of $\cC^0(X \times F, \bR)$ generated by all the $A_k$ for $k \geq 0$. For the set of sections as in Theorem~\ref{thm:StoneWeiestrass}, we take the following:
\begin{equation}\label{eqn:setofsectionsproofdensitylemma}
    \left\{ (x,u) \mapsto (s(x,u), u) \colon X \times F \to J^r\cE \times F \mid s \in \cC^0(F, \Gammahol(J^r\cE)) \right\}
\end{equation}
where again, for $s \in \cC^0(F, \Gammahol(J^r\cE))$, we see $s(\cdot,\cdot)$ as a map from $X \times F$ to $J^r\cE$ by adjunction.
We may now check the conditions of Theorem~\ref{thm:StoneWeiestrass}.
\begin{enumerate}
    \item Let $(x,u) \neq (x',u') \in X \times F$. Consider the first case where $x \neq x'$. For $k\geq 2$, $\cL^k$ is 2-very ample (see Example~\ref{example:jetampleness}). Hence there exists a section $s \in \Gammahol(\cL^2)$ such that $s(x) \neq 0$ and $s(x') = 0$. Then the map $(x,u) \mapsto |s(x)|^2$ is in $A_k$ and separates $(x,u)$ and $(x',u')$ as $|s(x)|^2 \neq 0$ and $|s(x')|^2 = 0$. In the other case where $x = x'$, we have that $u \neq u'$. By the $1$-very ampleness of $\cL$ we may choose $s \in \Gammahol(\cL)$ such that $s(x) = s(x') \neq 0$. Let $\rho : F \to \bR_+$ be a bump function such that $\rho(u) = 0$ and $\rho(u') = 1$. Then the map $(x,u) \mapsto |\rho(u)s(x)|^2$ is in $A_1$ and separates the points. Indeed it is vanishing at $(x,u)$ but non-vanishing at $(x',u')$.
    \item The second point is exactly what we have just proved in the first case of the first point above.
    \item It suffices to prove that the fibre of $J^r\cE$ above $x\in X$ is spanned by the sections $s(x)$ for $s \in \Gammahol(J^r\cE)$. This is implied by the $0$-jet ampleness of $\cE$ (see Example~\ref{example:jetampleness}).
\end{enumerate}
By construction, any element in the image of the map
\[
    \cC^0\big(F, \ \colim\limits_{k \to \infty} \cN(k)\big) \lra \cC^0\big(F, \ \Gamma_{\cC^0, \mathrm{ns}}\left(J^r\cE\right)  \big)
\]
is, by adjunction, in the $A$-module generated by the set~\eqref{eqn:setofsectionsproofdensitylemma}.
\end{proof}

\section{Applications}
\label{section:app}

\subsection{Non-singular sections of line bundles}

Our first application concerns the case of non-singular sections of line bundles, which was the starting motivation for this work. Here, a direct corollary of our main theorem reads as:

\begin{corollary}\label{cor:nonsingularsectionslinebundle}
Let $X$ be a smooth projective complex variety and $\cL$ be a very ample line bundle on it. Let $d \geq 1$ be an integer. The jet map
\[
    j^1 \colon \Gamma_{\mathrm{hol, ns}}\left( \cL^d \right) \lra \Gamma_{\cC^0, \mathrm{ns}}\left( J^1\cL^d \right)
\]
from non-singular holomorphic sections of $\cL^d$ to continuous never vanishing sections of $J^1\cL^d$, induces an isomorphism in homology in the range of degrees $* < \frac{d-1}{2}$.
\end{corollary}
\begin{proof}
It is a straightforward application of Theorem~\ref{thm:mainthm} by taking the admissible Taylor condition $\fT$ to be the zero section of $J^1\cL^d$ and recalling from Example~\ref{example:jetampleness} that if $\cL$ is very ample, then the tensor power $\cL^d$ is $d$-very ample.
\end{proof}

More interestingly, we can furthermore compute the stable rational cohomology. This agrees with a computation also made by Tommasi in~\cite{tommasi_stable_nodate}.
\begin{theorem}\label{thm:rationalcohomology}
Let $n = \dim_\bC X$ be the complex dimension of $X$. For $d \geq 1$, there is a rational homotopy equivalence:
\[
    \Gamma_{\cC^0,\mathrm{ns}}\left( J^1\cL^d \right) \overset{\simeq_\bQ}{\lra} \prod_{i=1}^{2n +1} K\left(H_{i-1}(X;\bQ), i\right).
\]
In particular, the rational cohomology of $\Gamma_{\cC^0, \mathrm{ns}}\left( J^1\cL^d \right)$ is given by the free commutative graded algebra
\[
    \Lambda\left(H^{*-1}(X;\bQ)\right)
\]
on the cohomology of $X$ shifted by one degree.
\end{theorem}

\begin{remark}\label{remark:cohomologicalstability}
This result implies in particular that the rational (co)homology of $\Gamma_{\mathrm{hol, ns}}\left( \cL^d \right)$ stabilises as $d \to \infty$. As we will see below, the integral cohomology does not stabilise in general.
\end{remark}
\begin{remark}
The stable cohomology only depends on the topology of $X$. This is in accordance with the analogies between topology and arithmetic and motivic statistics mentioned in the introduction. In both the results of Poonen and Vakil--Wood, the limit is expressed by a zeta function which only depends on $X$.
\end{remark}

\begin{example}
For $X = \bC\bP^n$ and $\cL = \cO(1)$, we find that the stable rational cohomology is the exterior algebra
\[
    \Lambda_\bQ(t_1, t_3, \ldots, t_{2n+1})
\]
where $t_i$ is in degree $i$. This agrees with the result of Tommasi in~\cite{tommasi_stable_2014}.
\end{example}

\begin{proof}[Proof of Theorem~\ref{thm:rationalcohomology}]
Recall that the non-singular sections of $J^1\cL^d$ are precisely the never-vanishing ones. We choose a Riemannian metric once and for all and denote by $\Sph\left(J^1\cL^d\right) \to X$ the unit sphere bundle of the vector bundle $J^1\cL^d$. We may scale a never vanishing section to have norm equal to $1$ (for the chosen metric) in each fibre. We thus obtain a homotopy equivalence:
\[
    \Gamma_{\cC^0, \text{ns}}\Big(J^1\cL^d\Big) \overset{\simeq}{\lra} \Gamma_{\cC^0}\Big( \Sph\left(J^1\cL^d\right)\Big).
\]
We now rationalise the sphere bundle in the following sense. By~\cite[Theorem 3.2]{llerena_localization_1985}, there is a fibration $S^{2n+1}_\bQ \to \Sph\left(J^1\cL^d\right)_\bQ \to X$ and a morphism of fibrations:
\begin{equation*}
\begin{tikzcd}
S^{2n+1} \arrow[d] \arrow[rr] &   & S^{2n+1}_\bQ \arrow[d] \\
\Sph\left(J^1\cL^d\right) \arrow[rr] \arrow[rd]       &   & \Sph\left(J^1\cL^d\right)_\bQ \arrow[ld]       \\
                              & X &                       
\end{tikzcd}
\end{equation*}
such that the map induced on the fibres $S^{2n+1} \to S^{2n+1}_\bQ \simeq K(\bQ, 2n+1)$ is a rationalisation. As $X$ is homotopy equivalent to a finite CW-complex and $S^{2n+1}$ is nilpotent (it is indeed simply connected), we may use~\cite[Theorem 5.3]{moller_nilpotent_1987} that shows that the map $\Sph\left(J^1\cL^d\right) \to \Sph\left(J^1\cL^d\right)_\bQ$ induces a map 
\[
    \Gamma_{\cC^0}\Big(\Sph\left(J^1\cL^d\right)\Big) \overset{\simeq_\bQ}{\lra} \Gamma_{\cC^0}\Big(\Sph\left(J^1\cL^d\right)_\bQ\Big)
\]
which is a rationalisation. (In general, one has to restrict to some path component. However both spaces are connected in our situation.) Now, oriented rational odd sphere bundles are classified by their Euler class (see e.g.~\cite[II.15.b]{felix_rational_2001}). In our situation, the orientation is induced from the canonical one on the complex vector bundle $J^1\cL^d$ and the Euler class vanishes for dimensional reasons. It follows directly that $\Sph\left(J^1\cL^d\right)_\bQ \to X$ is a trivial bundle. Therefore
\[
    \Gamma_{\cC^0}\Big(\Sph\left(J^1\cL^d\right)_\bQ\Big) \cong \map(X, K(\bQ,2n+1))
\]
where $\map(-,-)$ denotes the space of continuous functions with its compact open topology. Finally, in~\cite{thom_homologie_1956} (see also~\cite{haefliger_rational_1982} for an accessible reference), Thom proves that this mapping space is homotopy equivalent to a product of Eilenberg--MacLane spaces
\[
    \map(X, K(\bQ,2n+1)) \simeq \prod_{i=0}^{2n +1} K\left(H^{2n+1-i}(X;\bQ), i\right) \simeq \prod_{i=0}^{2n +1} K\left(H_{i-1}(X;\bQ), i\right)
\]
where the last equivalence comes from Poincaré duality. More precisely, he proves that if
\[
    \mathrm{ev} \colon \map(X,K(\bQ,2n+1)) \times X \lra K(\bQ,2n+1)
\]
is the evaluation map, and $\chi \in H^{2n+1}(K(\bQ,2n+1); \bQ)$ is the fundamental class, we may write
\[
    \mathrm{ev}^*(\chi) = \sum_i \chi_i
\]
where $\chi_i \in H^i(\map(X,K(\bQ,2n+1)); H^{2n+1-i}(X;\bQ))$. Then the projection 
\[
    \map(X,K(\bQ,2n+1)) \to K(H^{2n+1-i}(X;\bQ),i)
\]
is determined by the cohomology class $\chi_i$.
\end{proof}

\subsubsection{Geometric description of the stable classes and Mixed Hodge Structures}

As a Zariski open subset of the affine space $\Gammahol(\cL^d)$, the subspace $\Gamma_{\mathrm{hol,ns}}(\cL^d)$ inherits a structure of complex variety and its cohomology thus has a natural mixed Hodge structure. On the other hand, we may endow the stable cohomology computed in Theorem~\ref{thm:rationalcohomology} with a mixed Hodge structure defined as follows. Recall that the cohomology $H^*(X; \bQ)$ can be equipped with a mixed Hodge structure using the structure of complex variety on $X$, and denote by $\bQ(-1)$ the Tate--Hodge structure of pure weight $2$. By first tensoring these structures and then applying the symmetric algebra functor (see e.g.~\cite[Section 3.1]{peters_mixed_2008}), we obtain a mixed Hodge structure on the stable cohomology. In this section, we show the following.
\begin{proposition}
The morphism of Theorem~\ref{thm:rationalcohomology}
\[
    \Lambda\left(H^{*-1}(X;\bQ) \otimes \bQ(-1) \right) \lra H^*(\Gamma_{\mathrm{hol,ns}}(\cL^d); \bQ)
\]
is compatible with the mixed Hodge structures.
\end{proposition}

\begin{proof}
By the universal property of the (graded) symmetric algebra, it is enough to see that the morphism
\[
    H^{*-1}(X;\bQ) \otimes \bQ(-1) \lra H^*(\Gamma_{\mathrm{hol,ns}}(\cL^d); \bQ)
\]
respects the mixed Hodge structures. We will do this by giving a more geometric description of this map. Let 
\[
    \pi \colon \Gamma_{\mathrm{hol,ns}}(\cL^d) \times X \lra \Gamma_{\mathrm{hol,ns}}(\cL^d)
\]
be the trivial fibre bundle, and let
\[
    j \colon \Gamma_{\mathrm{hol,ns}}(\cL^d) \times X \lra J^1 \cL^d - \{0\}
\]
be the jet evaluation. By integrating along the fibres of $\pi$, we obtain in cohomology a morphism of mixed Hodge structures:
\[
    \pi_! \circ j^* \colon H^*(J^1 \cL^d - \{0\}) \otimes \bQ(n) \lra H^{*-2n}(\Gamma_{\mathrm{hol,ns}}(\cL^d)).
\]
The extra Tate twist $\bQ(n)$ comes from the definition of the Gysin map $\pi_!$ via Poincaré duality. (See~\cite[Corollary 6.25]{peters_mixed_2008}.) As the Euler class of the jet bundle vanishes for dimensional reasons, we compute that
\[
    H^*(J^1 \cL^d - \{0\}; \bQ) \cong H^*(X; \bQ) \otimes H^*(\bC^{n+1} - \{0\}; \bQ).
\]
Now $H^{2n+1}(\bC^{n+1} - \{0\}; \bQ) \cong \bQ(-n-1)$, so we have obtained a morphism of mixed Hodge structures:
\[
    \pi_! \circ j^* \colon H^*(X) \otimes \bQ(-1) \lra H^{*+1}(\Gamma_{\mathrm{hol,ns}}(\cL^d)).
\]
We claim that this coincides with the morphism given in Theorem~\ref{thm:rationalcohomology}. The proof is an exercise in algebraic topology and uses the description of the mapping space given at the end of the proof of Theorem~\ref{thm:rationalcohomology}. 
\end{proof}

\subsection{Integral homology and stability}

In this section, we focus on the special case where $X = \bC\bP^1$ and $\cL = \cO(1)$. That is, we study the space 
\[
    U_d := \Gamma_{\mathrm{hol, ns}}\left(\bC\bP^1, \cO(d) \right)
\]
of non-singular homogeneous polynomials in two variables of degree $d$. From Corollary~\ref{cor:nonsingularsectionslinebundle}, we know that the jet map
\[
    j^1 \colon U_d \lra \Gamma_{\cC^0, \mathrm{ns}}\left( J^1\cO(d) \right)
\]
induces an isomorphism in integral homology in the range of degree $* < \frac{d-1}{2}$. We prove that the section space on the right-hand side does not depend on $d \geq 1$, hence that the integral homology of $U_d$ stabilises.

\begin{theorem}
For $d \geq 1$, we have a homotopy equivalence
\[
    \Gamma_{\cC^0, \mathrm{ns}}\left( J^1\cO_{\bC\bP^1}(d) \right) \simeq \map(S^2,S^3).
\]
In particular
\[
    H_*(U_d; \bZ) \cong H_*(\map(S^2,S^3); \bZ)
\]
in the range of degrees $* < \frac{d-1}{2}$.
\end{theorem}

\begin{remark}
One may wonder how this reformulation helps in understanding the integral homology. Indeed, the homology of $\map(S^2,S^3)$ is still quite complicated. We therefore would like to point out that the homotopy theory of mapping spaces has a rich history of results. (See e.g.~\cite{smith_homotopy_2011} for a survey.)
\end{remark}

\begin{remark}
In the next section, we will show that one cannot expect integral homological stability in general. The case $X = \bC\bP^1$ should be seen as a very particular phenomenon.
\end{remark}

\begin{proof}
Recall from the proof of Theorem~\ref{thm:rationalcohomology} that we have to study continuous sections of the sphere bundle of the jet bundle:
\[
    S^3 \lra \Sph(J^1\cO_{\bC\bP^1}(d)) \lra \bC\bP^1.
\]
One sees that this bundle is classified by the second Stiefel--Whitney class of the jet bundle, i.e. the reduction modulo $2$ of its first Chern class. Using that $d \geq 1$ and \cite[Proposition 2.2]{di_rocco_line_2000}, we obtain an isomorphism of vector bundles:
\[
    J^1\cO_{\bC\bP^1}(d) \cong \cO_{\bC\bP^1}(d-1)^{\oplus 2}.
\]
We compute the first Chern class to be
\[
    c_1(J^1\cO_{\bC\bP^1}(d)) = c_1(\cO_{\bC\bP^1}(d-1)^{\oplus 2}) = 2 c_1(\cO_{\bC\bP^1}(d-1))
\]
so its reduction modulo $2$ vanishes regardless of $d$. As the sphere bundle was classified by this class, this shows that it is trivial. Therefore:
\[
    \Gamma_{\cC^0, \mathrm{ns}}\left( J^1\cO_{\bC\bP^1}(d) \right) \simeq \Gamma_{\cC^0}\left( \Sph(J^1\cO_{\bC\bP^1}(d)) \right) \simeq \map(S^2,S^3). \qedhere
\]
\end{proof}

\subsection{Integral homology and non-stability}

As we indicated in Remark~\ref{remark:cohomologicalstability}, the rational cohomology groups of the spaces $\Gamma_{\mathrm{hol,ns}}(\cL^d)$ stabilise. That is, for a fixed $i \geq 0$, the $i$-th rational cohomology group is independent of $d$ as long as $i \leq \frac{d-1}{2}$. In this section, to contrast with the very special case of the previous one, we show that one cannot expect integral stability in general. 

Let us fix some notation for the remainder of this section: $d \geq 1$ is an integer, $\cL$ is a very ample line bundle on a smooth projective complex variety $X$ and $n = \dim_\bC X$ is the complex dimension of $X$. As we will only be considering spaces of continuous sections, we will use $\Gamma$ as a shorthand for $\Gamma_{\cC^0}$.

The main result of this section is Theorem~\ref{thm:comparisoninfiniteloopspace} below. To show its computational potential, we will show the following:
\begin{proposition}\label{prop:zmodtwocomputation}
Let $d \geq 6$ be an integer. We have:
\[
    H_2(\Gamma_{\mathrm{hol,ns}}(\bC\bP^2, \cO(d)); \bZ/2) \cong \begin{cases} \bZ/2 & d \equiv 0 \mod 2 \\ 0 & d \equiv 1 \mod 2. \end{cases}
\]
Furthermore, the same result holds for the quotient $\Gamma_{\mathrm{hol,ns}}(\bC\bP^2, \cO(d)) / \bC^*$.
\end{proposition}

\subsubsection{A comparison map}

As stated in Corollary~\ref{cor:nonsingularsectionslinebundle}, we are reduced to studying the homotopy type of the space of continuous sections of the sphere bundle $\Sph(J^1\cL^d)$. Even though this is a purely homotopy theoretic problem, its resolution is quite hard. We will therefore ``linearise it'' in the homotopical sense using spectra. This is made precise in the following result:
\begin{theorem}\label{thm:comparisoninfiniteloopspace}
Let $TX$ be the tangent bundle of $X$, and let $X^{J^1\cL^d - TX}$ denote the Thom spectrum of the virtual bundle $J^1\cL^d - TX$ of rank $2$. There is a $2n$-connected map:
\[
    \Gamma(\mathrm{Sph}(J^1\cL^d)) \lra \Omega^{\infty+1} X^{J^1\cL^d - TX}.
\]
\end{theorem}

Our proof uses very lightly the theory of parametrised pointed spaces/spectra and is written using $\infty$-categories. We feel that the latter choice helps in conveying the main ideas more clearly. The unfamiliar reader is encouraged to think of bundles of pointed spaces/spectra, whilst resting assured that there exists a theory which renders all statements made here literally true. An encyclopedic reference is \cite{may_parametrized_2006}. As we shall only use basic adjunctions and Costenoble--Waner duality, we suggest to simply look at~\cite[Appendix A]{land_reducibility_2021} for a very readable introduction.

We denote respectively by $\catS_*$ and $\catSp$ the $\infty$-categories of pointed spaces and spectra. Likewise, we let $\catS_{*_{/X}} = \catFun(X, \catS_*)$ and $\catSp_{/X} = \catFun(X, \catSp)$ be the $\infty$-categories of parametrised pointed spaces/spectra over $X$. (In the definitions, $X$ is seen as an $\infty$-groupoid.) We let $r \colon X \to *$ be the unique map to the point. We will use the following three standard functors:
\begin{align*}
    &\text{the restriction functor: }& r^* \colon \catS_* \lra \catS_{*_{/X}}, \\
    &\text{its right adjoint: }& r_* \colon \catS_{*_{/X}} \lra \catS_*, \\
    &\text{its left adjoint: }& r_! \colon \catS_{*_{/X}} \lra \catS_*.
\end{align*}
The right and left adjoint are given respectively by right and left Kan extensions. In other words, $r_*$ takes the limit of a functor $F \in \catS_{*_{/X}} = \catFun(X, \catS_*)$, whilst $r_!$ takes its colimit. We will also use the analogous functors in the case of parametrised spectra with the same notation. It will be clear from the context which one we are using. The crucial fact for us is that for any bundle $Y \to X$ equipped with a section $s$ (this gives the data of a \emph{pointed} space over $X$), $r_*(Y)$ is the path component of $s$ in the section space.

As a last piece of notation, we will use $\Sigma^\infty_{/X} \dashv \Omega^\infty_{/X}$ to denote the infinite suspension/loop space adjunction between parametrised pointed spaces and spectra, and use $\Sigma^\infty \dashv \Omega^\infty$ to denote the usual adjunction in the unparametrised setting.

\bigskip

On our way to the proof of Theorem~\ref{thm:comparisoninfiniteloopspace}, we first make some formal observations. Loosely speaking, we would like to say that the section space of a fibrewise infinite loop space is the infinite loop space of the ``section spectrum''. This is made precise in the lemma below.
\begin{lemma}\label{lemma:infiniteloopspacesectionspace}
Let $Y \in \catS_{*_{/X}}$ be a parametrised space over $X$. We have a natural equivalence of pointed spaces:
\[
    \Omega^\infty r_* (\Sigma^\infty_{/X} Y) \simeq r_* (\Omega^\infty_{/X} \Sigma^\infty_{/X} Y).
\]
\end{lemma}

\begin{proof}
We use the Yoneda lemma and the adjunction $r^* \dashv r_*$. Let $Z \in \catS_*$ be a pointed space. We have:
\begin{equation*}
\begin{split}
    \map_{\catS_*}(Z,  \Omega^\infty r_* (\Sigma^\infty_{/X} Y)) &\simeq \map_{\catSp}(\Sigma^\infty, r_* (\Sigma^\infty_{/X} Y)) \\
        &\simeq \map_{\catSp_{/X}}(r^*\Sigma^\infty Z, \Sigma^\infty_{/X} Y) \\
        &\simeq \map_{\catSp_{/X}}(\Sigma^\infty_{/X} r^* Z, \Sigma^\infty_{/X} Y) \\
        &\simeq \map_{\catS_{*_{/X}}}(r^*Z, \Omega^\infty_{/X} \Sigma^\infty_{/X} Y) \\
        &\simeq \map_{\catS_*}(Z, r_*(\Omega^\infty_{/X} \Sigma^\infty_{/X} Y)).
\end{split}
\end{equation*}
Almost all manipulations follow from the standard adjunctions. The third equivalence uses the fact that $r^*\Sigma^\infty Z$ is the trivial parametrised spectrum with fibre $\Sigma^\infty Z$, hence is equivalent to $\Sigma^\infty_{/X} r^* Z$.
\end{proof}

We will need two more facts before proving Theorem~\ref{thm:comparisoninfiniteloopspace}. The first one is the following simple observation. If $V \to X$ is a vector bundle such that its associated sphere bundle $\Sph(V) \to X$ has a section $s$, then we may take the fibrewise infinite suspension $\Sigma^\infty_{/X}\Sph(V) \in \catSp_{/X}$ using $s$ to give a basepoint in each fibre. On the other hand, we could have taken the fibrewise one-point compactification and then suspend using the added point at infinity as a basepoint in each fibre. Up to a suspension, these are the same parametrised spectra.
\begin{lemma}\label{lemma:spherevscompactification}
Let $V \to X$ be a vector bundle with a non-vanishing section, and let $\Sph(V) \to X$ be its associated sphere bundle. Let $\bS_X^{V}$ denote the fibrewise infinite suspension of the fibrewise one-point compactification of $V$ (using the point at infinity as the basepoint in each fibre). Then:
\[
    \Sigma^\infty_{/X} \Sph(V) \simeq \Omega_X \bS_X^{V}
\]
where $\Omega_X$ denotes the desuspension in the category $\catSp_{/X}$.
\end{lemma}

\begin{proof}
Let us scale a non-vanishing section $s$ of $V$ to that it has image in the sphere bundle. We write $\mathrm{D}(V) \subset V$ for the unit disc bundle of $V$ which we point using $s$, and $V^+$ for the fibrewise one-point compactification. We obtain the lemma by applying the fibrewise infinite suspension $\Sigma^\infty_{/X}$ to the cofibre sequence $\Sph(V) \to \mathrm{D}(V) \to V^+$ of parametrised pointed spaces over $X$.
\end{proof}

Recall that the $\infty$-category $\catSp_{/X}$ is symmetric monoidal, with monoidal unit $\bS_X := r^*(\bS)$. (Here, and everywhere else, $\bS$ denotes the sphere spectrum.) The usefulness of the whole machinery set up so far is contained in the following result. A classical reference is \cite[Chapter 18]{may_parametrized_2006}. In the language of $\infty$-categories, one may read the second section of~\cite[Appendix A]{land_reducibility_2021}.
\begin{lemma}[Costenoble--Waner duality]\label{lemma:CostenobleWanerduality}
The Costenoble--Waner dualising spectrum of $X$ is $\bS_X^{-TX}$, the spherical fibration associated to the stable normal bundle of $X$. That is, we have an equivalence of functors:
\[
    r_*(-) \simeq r_!(- \otimes_{\bS_X} \bS_X^{-TX}).
\]
\end{lemma} 

We are now ready to prove the result announced at the beginning of this section.

\begin{proof}[Proof of Theorem~\ref{thm:comparisoninfiniteloopspace}]
We start by choosing once and for all a section $s$ of the sphere bundle $\mathrm{Sph}(J^1\cL^d)$, which provides us with a basepoint in every fibre. We may therefore apply the free infinite loop space functor $Q = \Omega^\infty \Sigma^\infty \colon \catS_* \to \catS_*$ fibrewise and obtain the following diagram of fibrations:
\begin{center}
\begin{tikzcd}
S^{2n+1} \arrow[d] \arrow[r]               & \Omega^\infty \Sigma^\infty S^{2n+1} \arrow[d] \\
\mathrm{Sph}(J^1\cL^d) \arrow[d] \arrow[r] & \Omega^\infty_{/X} \Sigma^\infty_{/X} \Sph(J^1\cL^d) \arrow[d]                                    \\
X \arrow[r, equal]           & X                                             
\end{tikzcd}
\end{center}
It is a standard fact that the map $S^{2n+1} \to \Omega^\infty \Sigma^\infty S^{2n+1}$ is $(4n+1)$-connected. Hence, on section spaces, the map
\[
    \Gamma(\Sph(J^1\cL^d)) \to \Gamma(\Omega^\infty_{/X} \Sigma^\infty_{/X} \Sph(J^1\cL^d))
\]
is $2n$-connected. (Notice that both spaces are connected, so the choice of $s$ was immaterial.) Using Lemma~\ref{lemma:infiniteloopspacesectionspace}, we obtain:
\[
    \Gamma(\Omega^\infty_{/X} \Sigma^\infty_{/X} \Sph(J^1\cL^d)) \simeq r_*(\Omega^\infty_{/X} \Sigma^\infty_{/X} \Sph(J^1\cL^d)) \simeq \Omega^\infty r_*(\Sigma^\infty_{/X} \Sph(J^1\cL^d)).
\]
We now make the purely formal following computation:
\begin{equation*}
\begin{split}
    r_*(\Sigma^\infty_{/X} \Sph(J^1\cL^d)) &\simeq r_!(\Sigma^\infty_{/X} \Sph(J^1\cL^d) \otimes_{\bS_X} \bS_X^{-TX}) \\
        &\simeq r_!(\Omega_X \bS_X^{J^1\cL^d} \otimes_{\bS_X} \bS_X^{-TX}) \\
        &\simeq r_!(\Omega_X \bS_X^{J^1\cL^d - TX}) \\
        &\simeq \Omega r_!(\bS_X^{J^1\cL^d - TX}) \\
        &\simeq \Omega X^{J^1\cL^d - TX}
\end{split}
\end{equation*}
where we used Lemma~\ref{lemma:CostenobleWanerduality} for the first equivalence, Lemma~\ref{lemma:spherevscompactification} for the second, and recognised that the value of $r_!$ on a spherical fibration is the associated Thom spectrum. Summing up, we get the result.

\end{proof}

\subsubsection{An example when $X = \bC\bP^2$}

To show how Theorem~\ref{thm:comparisoninfiniteloopspace} can be applied in practice, we use it to prove Proposition~\ref{prop:zmodtwocomputation}. We hope that this will convince the reader of the computational power of homotopy theoretic methods to study spaces of algebraic sections. 

Following Theorem~\ref{thm:comparisoninfiniteloopspace}, we should investigate $\Omega^{\infty+1} X^{J^1\cL^d - TX}$ when $X = \bC\bP^2$ and $\cL = \cO(1)$. Because $J^1\cL^d - TX$ is of rank $2$, the spectrum $\Omega X^{J^2\cL^d - TX}$ is $1$-connective and the bottom homotopy group is $\pi_1 \cong \bZ$ by Hurewicz theorem. We consider the fibration
\[
    F \lra \Omega^{\infty+1} X^{J^1\cL^d - TX} \lra S^1
\]
where $F$ is the homotopy fibre of the right-most map, which is taken to induce an isomorphism on $\pi_1$. A generator of $\pi_1(\Omega^{\infty+1} X^{J^1\cL^d - TX}) \cong \bZ$ gives a section of that fibration, and we obtain:
\[
    \Omega^{\infty+1} X^{J^1\cL^d - TX} \simeq S^1 \times F.
\]
In particular, $F$ is $2$-connective with $\pi_2(F) \cong \pi_2(\Omega^{\infty+1} X^{J^1\cL^d - TX})$. By Hurewicz theorem and the universal coefficient theorem, $H_2(F;\bZ/2) \cong H_2(F;\bZ) \otimes \bZ/2 \cong \pi_2(F) \otimes \bZ/2$. We thus wish to compute $\pi_2(\Omega^{\infty+1} X^{J^1\cL^d - TX})$, which we will do using the Adams spectral sequence at the prime $2$:
\[
    E^{s,t}_2 = \mathrm{Ext}^{s,t}_{\cA}\left(H^*(X^{J^1\cL^d - TX}; \bZ/2), \bZ/2\right) \implies \pi_*(X^{J^1\cL^d - TX})^\wedge_2.
\]
(Hence we will only compute the $2$-completed group, but this will be enough for our purposes.) The $E_2$-page is computed by knowing the cohomology $H^*(X^{J^1\cL^d - TX}; \bZ/2)$ as an algebra over the mod $2$ Steenrod algebra $\cA$. (See~\cite[Section 3.3]{beaudry_guide_2018} for a very readable introduction.) If $U$ denotes the Thom class of $J^1\cL^d - TX$, the classes in the cohomology of the Thom spectrum $X^{J^1\cL^d - TX}$ are given via the Thom isomorphism as $yU$ where $y \in H^*(X; \bZ/2)$. At the prime $2$, the Steenrod squares can then be computed from the formula:
\[
    \Sq^k(yU) = \sum_{i+j = k} \Sq^i(y) \Sq^j(U) = \sum_{i+j = k} \Sq^i(y) w_j U
\]
where $w_j$ is the $j$-th Stiefel--Whitney class of $J^1\cL^d - TX$. In our case, writing $\bZ/2[x]/(x^3)$ for the cohomology ring of $X = \bC\bP^2$, the total Stiefel--Whitney class is given by:
\[
    w(J^1\cL^d - TX) = \begin{cases}1 & d \equiv 0 \mod 2 \\ 1+x & d \equiv 1 \mod 2.\end{cases}
\]
We used the handy tool \cite{chatham_spectral_nodate} to compute the $E_2$-page for us, and obtained the following:
\captionsetup[subfigure]{labelformat=simple,singlelinecheck=on,textfont=normalfont,labelsep=period}
\begin{figure}[H]
\centering
\begin{subfigure}{0.49\textwidth}
  \centering
  \includegraphics[width=.8\linewidth]{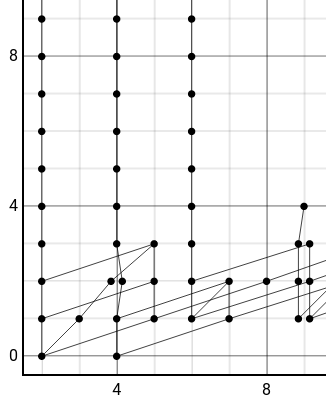}
  \caption{Case $d \equiv 0 \mod 2$}
\end{subfigure}
\begin{subfigure}{.49\textwidth}
  \centering
  \includegraphics[width=.8\linewidth]{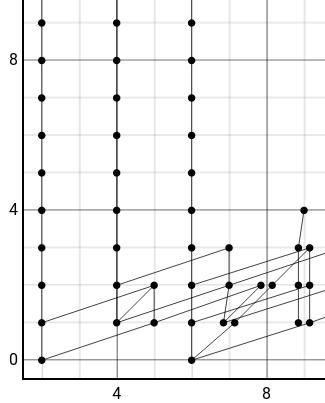}
  \caption{Case $d \equiv 1 \mod 2$}
\end{subfigure}
\caption*{Following the established convention, we use the Adams grading: the horizontal axis is indexed by $t-s$, the vertical one by $s$. Every dot represents a copy of $\bZ / 2$. The vertical lines represent multiplication by $h_0 \in \mathrm{Ext}^{1,1}_{\cA}(\bZ/2,\bZ/2)$. We suggest to the unfamiliar reader to look at \cite[Section 4.3]{beaudry_guide_2018} for more explanations.}
\end{figure}
From this, standard arguments about differentials (see e.g.~\cite[Section 4.8]{beaudry_guide_2018}) show that
\[
    \pi_3(X^{J^1\cL^d - TX})^{\wedge}_2 \cong \begin{cases}\bZ/2 & d \equiv 0 \mod 2 \\ 0 & d \equiv 1 \mod 2. \end{cases}
\]
Therefore
\[
    H_2(F;\bZ/2) \cong \pi_2(F) \otimes \bZ/2 \cong \pi_3(X^{J^1\cL^d - TX}) \otimes \bZ/2 \cong \begin{cases}\bZ/2 & d \equiv 0 \mod 2 \\ 0 & d \equiv 1 \mod 2. \end{cases}
\]
Using Künneth theorem, we obtain:
\[
    H_2(\Omega^{\infty+1} X^{J^1\cL^d - TX}; \bZ/2) \cong H_2(S^1 \times F; \bZ/2) \cong H_2(F;\bZ/2)
\]
which finishes the proof of Proposition~\ref{prop:zmodtwocomputation}.

\subsection{Stability of $p$-torsion}

In this final section, we study the $p$-torsion in the homology of the space $\Gamma_{\cC^0}\left( \Sph(J^1\cL^d) \right)$. On the one hand, we have just seen in Proposition~\ref{prop:zmodtwocomputation} that it depends on $d$ in general. On the other hand, the result below shows that when the prime $p$ is slightly bigger the dimension of $X$, the $p$-torsion is independent of $\cL$.

\begin{proposition}\label{prop:stableptorsion}
Let $X$ be a smooth complex projective variety of complex dimension $n$ and $\cL$ be a holomorphic line bundle on it. Let $p$ be a prime such that $p \geq n + 2$. Then the fibrewise $p$-localisation of the sphere bundle $\Sph(J^1\cL) \to X$ is trivial. In particular, we have an equivalence of $p$-local spaces
\[
    \Gamma_{\cC^0}\left( \Sph(J^1\cL) \right)_{(p)} \simeq \map(X, S^{2n+1}_{(p)}).
\]
\end{proposition}
As an immediate consequence, combining the proposition above with Corollary~\ref{cor:nonsingularsectionslinebundle} shows that the $p$-torsion in the homology of $\Gamma_{\mathrm{hol,ns}}(X; \cL^d)$ stabilises when $p \geq \dim_\bC X + 2$ and $d \to \infty$.

\bigskip

The proof uses the following result which we learned from \cite[Proposition~4.1]{bendersky_localization_2014}.
\begin{lemma}\label{lemma:ptorsionconnectedness}
For $p \geq \frac{k}{2} + \frac{3}{2}$, the space $\map_1(S^k_{(p)}, S^k_{(p)})$ of maps homotopic to the identity is $(k-1)$-connected.
\end{lemma}

\begin{proof}
The proof is given in \cite{bendersky_localization_2014}, but we sketch it here for convenience. We shall assume that $k$ is odd, as we will only use this case in this paper. Recall the evaluation fibration
\[
    \Omega^k_1 S^k_{(p)} \lra \map_1(S^k_{(p)}, S^k_{(p)}) \lra S^k_{(p)}.
\]
Using the associated long exact sequence of homotopy groups, it suffices to show that $\pi_i(\Omega^k_1 S^k_{(p)})$ vanishes for $i \leq k-1$. Using the assumption $p \geq \frac{k}{2} + \frac{3}{2}$, this follows from Serre's calculations on $p$-torsion in the homotopy groups of spheres.
\end{proof}

\begin{proof}[Proof of Proposition~\ref{prop:stableptorsion}]
Let 
\[
    S^{2n+1}_{(p)} \lra \Sph(J^1\cL)_{(p)} \lra X
\]
be the fibrewise $p$-localisation of $\Sph(J^1\cL) \to X$. By \cite[Theorem 5.3]{moller_nilpotent_1987}, we have a homotopy equivalence
\[
    \Gamma_{\cC^0}\left( \Sph(J^1\cL) \right)_{(p)} \simeq \Gamma_{\cC^0}\left( \Sph(J^1\cL)_{(p)} \right).
\]
As the sphere bundle is canonically oriented (using the complex orientation of $J^1\cL$), the fibration $\Sph(J^1\cL)_{(p)} \to X$ is classified by a map
\[
    X \lra B\map_1(S^{2n+1}_{(p)}, S^{2n+1}_{(p)}).
\]
By Lemma~\ref{lemma:ptorsionconnectedness}, the codomain of that map is $(2n+1)$-connected. As the domain has real dimension $2n$, the classifying map must be null-homotopic thus showing that the fibration is trivial.
\end{proof}

\bibliography{references.bib} 
\bibliographystyle{alpha}

\end{document}